\documentclass{amsart}

\usepackage{amsmath}

\usepackage{euscript}
\usepackage{mathrsfs}
\usepackage{bbold}
\usepackage{amssymb}
\usepackage{mathabx}
\usepackage{tikz}
\usetikzlibrary{arrows}
\usetikzlibrary{calc}

\usepackage{float}

\RequirePackage{color}
\definecolor{myred}{rgb}{0.75,0,0}
\definecolor{mygreen}{rgb}{0,0.5,0}
\definecolor{myblue}{rgb}{0,0,0.65}

\RequirePackage{ifpdf}
\ifpdf
  \IfFileExists{pdfsync.sty}{\RequirePackage{pdfsync}}{}
  \RequirePackage[pdftex,
   colorlinks = true,
   urlcolor = myblue, 
   citecolor = mygreen, 
   linkcolor = myred, 
   pagebackref,
   bookmarksopen=true]{hyperref}
\else
  \RequirePackage[hypertex]{hyperref}
\fi

\RequirePackage{ae, aecompl, aeguill} 

\newcommand{\nc}{\newcommand} \newcommand{\renc}{\renewcommand}

    \def\RM{{\mathbb{R}}}

    \def\ZM{{\mathbb{Z}}}

    \def\FC{{\mathcal{F}}}
    \def\GC{{\mathcal{G}}}

    \def\RC{{\mathcal{R}}}

\def\e{\varepsilon}

\def\XSS{{\mathscr{X}}}

\nc{\todo}[1]{ {\color{red}XXX #1 XXX}}

\def\to{\rightarrow}

\def\longto{\longrightarrow}

\def\onto{\twoheadrightarrow}
\nc{\triright}{\stackrel{[1]}{\to}}
\nc{\longtriright}{\stackrel{[1]}{\longto}}

\nc{\Hb}{H^\bullet}

\nc{\Br}{\mathcal{B}}
\nc{\HotRR}{{}_R\mathcal{K}_R}
\nc{\HotR}{\mathcal{K}_R}
\nc{\excise}[1]{}
\nc{\defect}{\text{df}}
\nc{\h}[1]{\underline{H}_{#1}}

\nc{\Ga}{\mathbb{G}_a} 
\nc{\Gm}{\mathbb{G}_m} 

\nc{\Perv}{{\mathbf{P}}}

\nc{\IH}{{\mathrm{IH}}}

\nc{\ic}{\mathbf{IC}}

\nc{\gl}{{\mathfrak{gl}}}
\renc{\sl}{{\mathfrak{sl}}}
\renc{\sp}{{\mathfrak{sp}}}

\renc{\Im}{\textrm{Im}}

\nc{\HBM}{H^{BM}}

\nc{\St}{\mathrm{St}}
\nc{\rot}{\mathrm{rot}}
\nc{\ext}{\mathrm{ext}}
\nc{\Tilt}{\mathrm{Tilt}}
\nc{\gen}{\mathrm{gen}}
\nc{\Graph}{\mathrm{Graph}}

\newcommand{\into}{\hookrightarrow}

\nc{\simto}{\stackrel{\sim}{\to}}
\nc{\simfrom}{\stackrel{\sim}{\leftarrow}}

\nc{\gbmod}{\mathrm{-gmod-}}

\nc{\gmod}{\mathrm{-gmod}}

\nc{\Parity}{\mathrm{Parity}}

\nc{\mult}{\mathrm{mult}}

\nc{\Hecke}{\textrm{H}}

\nc{\geom}{\mathrm{geom}}
\nc{\Soe}{\mathrm{Soe}}
\nc{\Abe}{\mathrm{Abe}}
\nc{\diag}{\mathrm{diag}}

\nc{\fin}{\textrm{finite}}

\nc{\reflect}{\RC}
\nc{\Chi}{\XSS}
\nc{\pt}{\mathrm{pt}}
\nc{\odd}{\textrm{odd}}
\nc{\even}{\textrm{even}}

\nc{\weights}{\textrm{weights}}

\renewcommand{\top}{\text{top}}
\renewcommand{\bot}{\text{bot}}
\DeclareMathOperator{\width}{width}
\DeclareMathOperator{\dist}{dist}
\DeclareMathOperator{\height}{ht}
\nc{\adhoc}{\textrm{adhoc}}
\nc{\fund}{\textrm{fund}}
\newcommand{\verteq}{\rotatebox{90}{$\,=$}}

\newtheorem{thm}{Theorem}[section]
\newtheorem{lem}[thm]{Lemma}

\newtheorem{prop}[thm]{Proposition}
\newtheorem{cor}[thm]{Corollary}

\theoremstyle{definition}

\theoremstyle{remark}
\newtheorem{remark}[thm]{Remark}

\newif\ifdraft\draftfalse
\ifdraft
\newcommand{\change}[1]{\textcolor{blue}{#1}}
\else
\newcommand{\change}[1]{#1}
\fi

\title[]{Drums of high width}

\author[]{Alex Davies}
\address{Google DeepMind.}
\email{adavies@google.com }

\author[]{Prateek Gupta}
\address{Google DeepMind.}
\email{prateek.gupta@exeter.ox.ac.uk}

\author[]{S\'ebastien Racani\`ere}
\address{Google DeepMind.}
\email{sracaniere@google.com }

\author[]{Grzegorz Swirszcz}
\address{Google DeepMind.}
\email{swirszcz@google.com}

\author[]{Adam Zsolt Wagner}
\address{Worcester Polytechnic Institute, USA.}
\email{zadam@wpi.edu}

\author[]{Theophane Weber}
\address{Google DeepMind.}
\email{theophane@google.com }

\author[]{Geordie Williamson}
\address{University of Sydney, Australia.}
\email{g.williamson@sydney.edu.au}
\thanks{GW is supported by ARC grant DP230102982.}

\begin{document}

\begin{abstract}
  We provide a family of $5$-dimensional prismatoids
  whose width grows linearly in
  the number of vertices. This provides a new infinite family of
  counter-examples to the Hirsch conjecture whose excess width grows
  linearly in the number of vertices, and answers a question of Matschke, Santos
and Weibel.
\end{abstract}

\maketitle




\section{Introduction} \label{sec:intro}

Let $P$ be a convex polytope. The graph of $P$ has nodes and edges
corresponding to the vertices and edges of $P$. It is a fundamental
unsolved problem to determine the diameter of this
graph: how many steps does one need in order to walk between any two
vertices, if one may only traverse edges.

This problem has deep connections
to the complexity of the simplex method in linear
programming. Indeed, lower bounds on the diameter of a polytope give
lower bounds on the complexity of the simplex method. Determining the
complexity of linear programming is one of Smale's 18 problems, and is
still unsolved.\footnote{See \cite[\S 1]{Santos} for some interesting
  history on the diameter  of polytopes and its connection to the
  complexity of linear   programming.} 

In 1957, Hirsch conjectured that the diameter of any convex polytope is
bounded by $n-d$, where $n$ is the number of facets and $d$ is the
dimension. This conjecture provides a simple linear bound for the diameter
of a polytope, and would show that the diameter cannot provide a useful
lower-bound for the complexity of the simplex method

In a breakthrough paper in 2011, Santos \cite{Santos} provided a
counter-example to the Hirsch conjecture. His counter-example is in
dimension 43 and has 86 facets, and the Hirsch bound fails by
1. Santos' counter-example re-opens the door to the possibility that
the diameter of polytopes might lead to interesting lower-bounds for the
complexity of the simplex method.

Motivated by the ``$d$-step theorem'' of Klee and Walkup
\cite{KleeWalkup}, Santos introduced the fundamental concept of a
\emph{drum}.\footnote{\emph{Drum} is our terminology, Santos called them
  \emph{prismatoids}} These are
polytopes with two parallel facets containing all vertices. A
remarkable theorem of Santos shows that interesting drums in small
dimension (e.g. 5), may be used to construct interesting polytopes in high
dimension (e.g. 43 in Santos' famous example).

Experience suggests that the difficulty of gaining human understanding
of polytopes grows exponentially in the dimension. This is one reason
the drum concept is so useful---it allows one to compress
high-dimensional difficulty into comparatively low dimension.
More generally, the
work of Santos and subsequent work of Maschke, Santos and
Weibel \cite{MSW} suggest that the study of drums may provide insights
into the difficult question of the diameter of general polytopes.

Despite considerable interest in this problem, there remain very few
examples of polytopes breaking the Hirsch bound. All known examples 
arise from interesting 5-dimensional drums. Such ``non-Hirsch
drums'' appear to have remarkable structure, and are extremely rare.
In \cite{Santos},
Santos provides examples with enormous numbers of facets to show that
the Hirsch conjecture can fail by any constant amount. More recently,
Maschke, Santos and Weibel \cite{MSW} constructed more examples,
showing that the excess width of drums (see below) can grow like the square root of the number of
facets. In this paper we provide a family of 5-dimensional drums whose excess
width grows linearly in the number of facets, answering a question
posed in \cite{MSW}. Up to a constant factor, this is best possible. We
are also able to uncover the mechanism leading to such examples.

\subsection{Main theorem} As is often the case in the study of
polytopes, it is useful to dualize. To a convex
polytope $P$ we may also associate its facet-ridge graph which has nodes
given by the facets (=codimension 1 faces) of $P$ and edges
corresponding to ridges (=codimension 2 faces) of $P$. The facet-ridge
graph of $P$ agrees with the (vertex-edge) graph of the dual polytope.
In this language the Hirsch conjecture states that the diameter of the
facet-ridge of a convex polytope is bounded by $n-d$ where $n$ is the
number of vertices of $P$ and $d$ is the dimension.

As discussed above, in this paper we consider \emph{drums}. These are
polytopes $D$ that posess two fixed parallel facets $D^-$ and $D^+$ (the \emph{drum skins}) which together contain all
vertices of $D$. (For example, the cube and octahedron are both drums,
with respect to any pair of opposite faces.) The width of a drum is
the distance between the two drum skins in the facet-ridge graph. (Of
course, the width of a drum provides a lower bound on the diameter of
its facet-ridge graph.)

Santos explained how to start with a $d$-dimensional drum with
$n$ vertices of width $\ge w$ and produce another drum with $2(n-d)$ vertices in
dimension $n-d$ of diameter $\ge w+n-2d$. Thus, the Hirsch conjecture
implies that drums in dimension $d$ have width at most $d$. Santos
constructs his counter-example starting from a 5-dimensional drum with
48 vertices of width $6 > 5$.

More generally, Maschke, Santos and Weibel \cite{MSW} suggested that the study of
the width of drums might provide insights on the diameter of general
polytopes. Thus, one is led to consider the function
\[
f_d(n) = \max_{\text{$d$-dimensional drums $D$ with $n$-vertices}}
\width(D) -d.
\]
We call this function the \emph{excess width}. The (disproved) Hirsch
conjecture would have implied that this function is identically zero.
All interesting examples so far involve 5-dimensional drums. The best results so
far are due \cite{MSW} and show that $f_5$ grows at least as fast as
the square root of $n$.

Here we produce a new family of 5-dimensional drums $D_k$ parametrized
by $k = 1, 2, \dots$. They have $16k + 24$ vertices. Our main theorem is:

\begin{thm} \label{thm:main}
$D_k$ has width $\ge 5+k$.
\end{thm}

Thus our results prove that $f_5$ grows linearly in $n$. The question
of whether this is possible was raised in \cite{MSW}. Note also that
we hit (up to constants) the
bounds established by \cite{MSW} for 5-dimensional drums\footnote{Indeed, for any fixed constant $d$, the function $f_d(n)$ is at most linear in $n$.}.

To translate this result back to the world of the Hirsch conjecture,   Santos' Strong $d$-step theorem for spindles, which was mentioned above, implies the following:
\begin{cor} \label{cor:main_hirsch_translation}
    For any integer $k\geq 1$, there exists a polytope $P_k$ of
    dimension $d=16k + 19$ with $2d$ facets, that has diameter at
    least $  d+k $. That is, the Hirsch bound fails by $k$.
  \end{cor}
  
In \cite[\S 6]{Santos}, Santos introduces the \emph{Hirsch excess} of
a (non-Hirsch) polytope which is defined as
\[
\frac{\delta}{n-d} - 1
\]
where $\delta$ is its diameter, $n$ the number of facets, and $d$ the
dimension. It is a useful measure of ``how much'' a polytope violates
the Hirsch conjecture. One computes easily that the Hirsch excess of
the polytopes arising from our family of drums is
\[
\frac{k}{16k + 19}
  \]
Thus, as $k \to \infty$ the Hirsch excess of our family approaches
$1/16$, which is the largest excess known (to the best of our knowledge).
  

\subsection{Symmetric search} Our family of drums (given in \S 3)
might seem arbitrary at first. Here we give some idea of how we
discovered this family.

Let us first recall Santos' construction \cite{Santos} of his 5-dimensional drum of
width $6$, which he uses to produce the first counter-example to the
Hirsch conjecture. Let $\Gamma$ denote the finite group of orthogonal
transformations of $\RM^5$ generated by the matrices
\[
\left (  \begin{matrix}
  \pm 1 & 0 & 0 & 0 & 0 \\
  0 & \pm 1 & 0 & 0 & 0 \\
  0 & 0 & \pm 1 & 0 & 0 \\
  0 & 0 & 0 & \pm 1 & 0 \\
  0 & 0 & 0 & 0 & 1 \\
\end{matrix} \right )
\quad \text{and} \quad
\tau =\left (  \begin{matrix}
  0 & 0 & 0 & 1 & 0 \\
  0 & 0 & 1 & 0 & 0 \\
  1 & 0 & 0 & 0 & 0 \\
  0  & 1 & 0 &0  & 0 \\
  0 & 0 & 0 & 0 & -1 \\
\end{matrix} \right ).
\]
Now, start with the five vectors
\begin{equation} \label{santos vertices}
(18,0,0,0,1), (0,0,45,0,1), (15,15,0,0,1), (0,0,30,30,1),
(10,0,0,40,1)
\end{equation}
and consider all possible vectors obtained from one of these vectors by
application of $\Gamma$. One obtains in this way 48 vectors, all of
whose last coordinate is $\pm 1$. Hence, if $D$ denotes the convex hull of these
points then one obtains a drum with respect to the two faces $D^-$
(resp. $D^+$) where the last coordinate is $-1$ (resp. $+1$). Santos
proves that $D$ has width $6$.

It is striking here that a rather complicated drum $D$ is generated by
a rather small \emph{motif}, namely the list of vectors \eqref{santos vertices}. We
first tried random search over tuples of vectors whose entries are
given by vectors whose entries are integers between $0$ and
$50$.\footnote{More precisely, our search enforced some sparsity. We noticed
the vectors in Santos' example have lots of zeros, and enforced this
by first choosing a vanishing pattern (e.g. $(A, 0, 0, 0, 1)$ or $(A,
0, 0, B, 1)$) and then choosing entries for $A$, $B$ etc.} Our first
observation was that, although rare, we were able to find other
width 6 drums in this way, which are genuinely different to Santos'
construction. This is an illustration of an important principle in
search problems, namely that searching for symmetric solutions is
often better than brute-force (see e.g. \cite{sym}).

Our second observation is that sampling entries of greatly
differing scales was much more likely to produce counter-examples. For
example, the motif
\begin{equation}   \label{eq:vecs2} (0,0,3,3,1),   (98,0,1,0,1), (100, 0, 0, 0, 1), 
(75,75,0,0,1).
\end{equation}
generates a drum of width $6$. (In fact, this is the first element of
our family.) Note the differing scales here: the first two non-zero coordinates (e.g. $100,
98, 75$) are of the order of 100, whereas the second two non-zero
coordinates ($1$, $3$) are of the order of $1$. Random sampling with
the first two coordinates ``big'' and the second two ``small''
appeared significantly more likely to produce drums of width $6$.\footnote{Drums of this form have a
striking geometric form, which we refer to as a ``double pancake'', see
Remark \ref{rem:double pancake}.}

Our next step was to try to
take a motif generating a drum of width $k$, and find a new motif by
adding a single vector in order to produce a drum of width $k+1$. This
was performed by a combination of human and computer
searches. After considerable effort, we found a family of motifs which led to
drums of widths 8, 9, 10, all the way to 25. At this point we were confident
that the construction worked, the proof only came later.

Interestingly, we found several examples of motifs 
which could be modified by the addition of additional vertices to
produce drums of widths $6,7,8,\dots$, however at some point it
appeared impossible to add a single vector to increase the
width. After considerable searches, the only family which produces
widths $5+k$ for all $k$ was the example of this paper.

\begin{remark} \label{rem:constants}
This entire paper is simply the discussion of an interesting
family of polytopes. We have not attempted to make the description of
our polytopes as general as possible. We have simply found constants which
do the job. The interested reader is encouraged to try to modify our
constructions, or obtain constants which are the most general
possible.
\end{remark}

\begin{remark} As we alluded to above, work on this paper relied
  heavily on extensive computations. Although we do not rely on any of
  the computations for the results of this paper, the interested
  reader may (like us) wish to experiment. For this
  reason we have prepared a colab \cite{colab} which contains some basic functions
  to compute with the objects of this paper. Throughout the paper we
  include several remarks with pointers to computations and examples in the colab.
\end{remark}

\subsection{Idea of the proof} Each of our polytopes $D_k$ is 
complicated, and we are unable to get a full 
description of their combinatorial structure. In order to prove that
they have width $\ge 5 + k$ we exploit a strategy developed in
\cite{Santos, MSW}. Namely, we obtain an explicit combinatorial
description of the top and bottom drum skins (which are much simpler
$4d$ polytopes), and then try to glean as much as we can concerning
their interaction. As in \cite{Santos,MSW} we make heavy use of the
symmetry group $\Gamma$ to reduce computations.

A key tool developed in \cite{Santos} and extended in \cite{MSW} is
that one may compute the width of a $d$-dimensional drum in terms of a
topological data (a ``pair of geodesic maps'') in the
$S^{d-2}$-sphere. The pair of geodesic maps describes combinatorics of
the interaction of the top and bottom drum skins. Then \cite[\S 2.1]{MSW} explained how to distill
out of this pair of geodesic maps a bipartite graph (the ``incidence
pattern''), whose nodes are the facets of the top and bottom drum
skins. They prove that the drum provides a counter-example to the
Hirsch conjecture if and only if this graph is free of oriented
$2$-cycles \cite[Proposition 2.3]{MSW}.

We give a different point of view on the incidence pattern.
We explain in \S\ref{sec:fv} how one may compute it in terms of linear
programming problems associated to the top and bottom drum
skins. Roughly speaking, each facet of the top drum skins determines a
linear programming problem on the bottom drum skin, and
conversely. The solutions of these problems 
determines the incidence pattern. (We rename this map the
``facet-vertex map''.)

Considerations of the incidence pattern makes it very believable that
our drums are of width $\ge 5 + k$. However, this alone does not appear
enough to conclude the proof. We conclude by a careful study of the
interaction between ridges in the top drum skin and edges in the bottom drum skin.

\change{
  \subsection{A comment on terminology} In this
  work we call ``drum'' what Santos \cite{Santos} calls ``prismatoid''. The reader
  may wonder why we decided to deviate from accepted nomenclature. We
  find the notation $P$ (for a general \emph{P}olytope) and $D$ (for a
  \emph{D}rum)  natural and suggestive. We also find the terminology
  ``drum'' and ``drum skin'' evocative of the striking geometric structure.
}

\change{ 
\subsection{Acknowledgements} We would like to thank the
  referees for their comments, which led to an improvement of this work.}

\section{Drums, skins and widths}

This is the theoretical foundation of this paper. We review
some polytope basics and fix our notation. We then introduce the pair
embedding and facet-vertex map, and explain how they can be used to
bound the width of a drum from below.

\subsection{Conventions and notation for polytopes}

We work throughout with polytopes $P$ within a $d$-dimensional affine
space $\RM^d$. \emph{Polytope} means convex polytope.

We use \emph{functional} throughout to mean affine linear
functional. Recall that a \emph{face} of a polytope $P$ is a subset of
the form $\{ p \in P \; | \; f(p) = 0 \}$ for some functional $f$ which
is $\ge 0$ on $P$. The dimension of a face is the dimension of its
affine span.

By \emph{facet}, \emph{ridge}, \emph{edge} and \emph{vertex} we mean
face of codimension $1$, codimension $2$, dimension $1$ and
dimension $0$ \change{faces} respectively.

Given a polytope $P$ with vertices $v_1, \dots, v_m$ we will often
abuse notation and identify a face $F$ with the subset of vertices
which it contains. Thus, ``the edge $\{ v_1, v_2\}$" really means
``the unique edge which contains $v_1$ and $v_2$".

For any facet $F$ there is a functional $f_F$ which vanishes
on $F$ and is $\ge 0$ on $P$. This functional (which
is unique up to positive scalar) is the functional \emph{defining} $F$.

Given a polytope $P$, a facet $F$ and a point $p$ outside $P$ (i.e. $p
\notin P$) we say that $p$ is \emph{visible} from $F$ if $f_F(p) <
0$. Intuitively, $P$ is a piecewise linear (and opaque) planet, and we
are a tiny being standing on the facet $F$. Visible points are all
those points that we can see above the horizon of our planet.

Quite a few graphs come up in this paper. There is a conflict of
terminology: vertices are often used to refer to the nodes of a graph,
which might be e.g. facets of a polytope. In order to try to avoid
becoming horribly confused we exclusively use \emph{vertex} to refer
to a $0$-dimensional face of a polytope, and \emph{node} to refer to
the node of a graph.

\subsection{Conventions and notations for drums}

A \emph{drum} is a polytope $D$ with two distinguished parallel facets $D^-$ and
$D^+$ which together contain all vertices of $D$. We often call $D^-$ and $D^+$ the \emph{bottom} and \emph{top
  drum skins}. (Thus, what we call a drum, Santos \cite[Definition
2.5]{Santos} calls a \emph{prismatoid}. In Santos' language,
prismatoids are the polar duals of \emph{spindles}. We will not need
the language of spindles in this paper.)

 As pointed out in \cite[\S 2.2]{Santos}, the requirement that the
 drum skins be parallel is not particularly important. One could
 instead require them simply to be disjoint, in which case a \change{projective}
 transformation can always be chosen to make them parallel.

Let us emphasise that the top and bottom drum skins are part of the
data of a drum (i.e. a drum is a polytope $D$, together with two
distinguished facets $D^-$ and $D^+$). For us, it will be convenient
to always assume that our polytope is embedded in $\RM^d$ in such a
way that the bottom and top drum skins are ``at height $-1$ and $1$
respectively''. More precisely, if $z$ denotes the last coordinate in
$\RM^d$ we assume that
\[
D^- = D \cap \{ z = -1\} \quad \text{and} \quad D^+ = D \cap \{ z = 1\} .
\]

A drum is \emph{simplicial} if all facets other than the drum skins
are simplices. We make the following important assumption:
\begin{equation}
  \label{eq:simplicial_assumption}
\begin{array}{c}  \text{All drums considered in this paper are assumed simplicial.} \end{array}
\end{equation}

\begin{remark}
  One may understand simplicial polytopes are polytopes which are
  generic in their vertex description: any small
  peturbation of their vertices results in a combinatorially
  equivalent polytope. Similarly, one may understand simplicial drums
  as drums which are generic in their vertex description: any small
  peturbation of their vertices \emph{preserving the drum skins}
  results in a combinatorially equivalent drum. Note that a simplicial
  drum is never a simplicial polytope unless its drum skins are simplices.
\end{remark}


\subsection{The pair embedding} \label{sec:pair} In this section we provide a way to bound the
width of a drum from below in terms of certain graphs associated to
the top and bottom drum skins.

Consider a polytope $P$. We can associate two important graphs to $P$:
\begin{enumerate}
\item \emph{The facet-ridge graph $\FC(P)$:} Nodes are facets of $P$, and
  two facets $F$ and $F'$ are joined if they have a ridge in common.
\item \emph{The face graph $\Delta(P)$:} Nodes are faces of $P$ of any
  dimension. There is an edge between two nodes if they are
  incident and their dimension differs by $1$. The face graph
  is graded by dimension of the face. We include the empty face, which
  has dimension $-1$.
\end{enumerate}

Now consider a $d$-dimensional drum $D$ with bottom drum skin $D^-$ and top drum
skin $D^+$. By definition, the \emph{width} of $D$ is
\begin{equation}
\width(D) = \dist_{\FC(D)}(D^-, D^+).\label{eq:width1}
\end{equation}
(Here $\dist$ has the usual
  meaning in graph theory: $\dist(v,v) = 0$, $\dist(v,v') = 1$ if and
  only if $v \ne v'$ and $v$ and $v'$ are incident etc.)

It makes things a little easier below to ignore the top and bottom drum
skins, as they are of a very different character to the rest of the
facets of the drum. The following definition only makes sense for drums:
\begin{enumerate}
\item \emph{The trimmed facet-ridge graph $\FC^{t}(D)$:} The full
  subgraph of $\FC(D)$ consisting of facets $\ne D^-, D^+$.
\end{enumerate}

Because of our assumption \eqref{eq:simplicial_assumption} we can
grade $\FC^{t}(D)$ by dimension of the intersection of a facet with the
top drum skin (which can take values $0, 1, \dots, d-2$). We refer to
this grading as the \emph{height}, and denote it $\height F$. Because $D^-$
(resp. $D^+$) is incident to every facet at height $0$
(resp. $d-2$) we can rephrase the width in the trimmed
facet-ridge graph as follows:
\begin{equation}\label{eq:width2}
\width(D) = \min_{F, F', \atop \height F = 0, \height F' = d-2}
\dist_{\FC^{t}(D)}(F, F') + 2
\end{equation}

Any facet $F$ of $D$ intersects the top and bottom drum skins $D^-$
(resp. $D^+$) in faces $F_{\bot}$ (resp. $F_{\top}$) of dimension $a$
(resp. $b$) with $a + b = d-2$. The facet $F$ is determined uniquely
by $F_\bot$ and $F_\top$.
In
this way we obtain an injection (a priori of sets):
\begin{align*}
  \rho: \FC^{t}(D) & \into \Delta(D^-) \times \Delta(D^+) \\
  F &\mapsto (F_{\bot}, F_\top)
\end{align*}
This map will play a crucial role below. We call it the
\emph{pair embedding}.

Let us make the product $\Delta(D^-) \times \Delta(D^+)$ into a graph
via the (box) product: nodes are given by the Cartesian product of
the nodes of $\Delta(D^-)$ and $\Delta(D^+)$ and $(u_1, v_1)$ and
$(u_2, v_2)$ are connected by an edge if and only if either $u_1 =
u_2$ and $v_1, v_2$ are incident, or $u_1$ and $u_2$ are incident and
$v_1 = v_2$. The importance of the pair embedding is due to the following
proposition, which bounds distance in the trimmed facet ridge graph in
terms of the image under the pair embedding.

\begin{prop}\label{prop:bound_width} For any two faces $F, F' \in \FC^{t}(D)$ we have
  \[
\dist_{\FC^{t}(D)} (F, F') \ge \frac{1}{2} \dist_{\Delta(D^-) \times \Delta(D^+)} (
\rho(F), \rho(F')).
    \]
\end{prop}

\begin{proof} Suppose that $F, F'$ are incident faces in $\FC^{t}(D)$. It
  is enough to show that the distance between $\rho(F)$ and $\rho(F')$
  in $\Delta(D^-) \times \Delta(D^+)$ is $\le 2$.

  For incident $F, F'$ there are three possibilities:
\begin{enumerate}
\item \emph{$F$ and $F'$ intersects $D^+$ (equivalently $D^-$) in faces of
  different dimensions.} In this case if
  \[
\rho(F) = (F_\bot, F_\top) \quad \text{and} \quad \rho(F') = (F'_\bot, F'_\top) 
\]
then $F_\bot$ and $F'_\bot$ (resp. $F_\top$, $F'_{\top}$) are incident
in $\Delta(D^-)$ (resp. $\Delta(D^+)$). Hence
\[
\dist_{\Delta(D^-) \times \Delta(D^+)} ( \rho(F), \rho(F')) = 2.
  \]
\item \emph{$F$ and $F'$ have the same intersection with $D^-$.} In this case if
  \[
\rho(F) = (F_\bot, F_\top) \quad \text{and} \quad \rho(F') = (F'_\bot, F'_\top) 
\]
then $F_\bot = F'_\bot$, and $F_\top$ and $F'_{\top}$ share a common
codimension 1 face, and hence are of distance $2$ in $\Delta(D^+)$.
\item \emph{$F$ and $F'$ have the same intersection with $D^+$.} In this case if
  \[
\rho(F) = (F_\bot, F_\top) \quad \text{and} \quad \rho(F') = (F'_\bot, F'_\top) 
\]
then $F_\top = F'_\top$, and $F_\bot$ and $F'_{\bot}$ share a common
codimension 1 face, and hence are of distance $2$ in $\Delta(D^-)$.
\end{enumerate}
This completes the proof.
\end{proof}

Proposition \ref{prop:bound_width} is not quite strong enough for our
needs. We need a slight strengthening, which is obtained by observing
that some facets of $D^-$ and $D^+$ never occur in the image of the
pair embedding. More precisely, consider the full subgraphs with nodes
\begin{gather*}
\Delta_D(D^+) = \{ F_\top \; | \; F \text{ facet of $D$} \} \subset \Delta(D^+),\\
\Delta_D(D^-) = \{ F_\bot \; | \; F \text{ facet of $D$} \}\subset \Delta(D^-).
\end{gather*}
In other words, the nodes of $\Delta_D(D^+)$ consists of all faces of
$D^+$ which occur as $F_\top$ for some facet $F$ of $D$, and similarly
for $\Delta_D(D^-)$.

Because the image of the pair embedding lands in the full subgraph
\[
\Delta_{D}(D^-) \times \Delta_{D}(D^+) \subset \Delta(D^-) \times \Delta(D^+) \]
we have:

\begin{prop} \label{prop:pair_bound}
   For any two faces $F, F' \in \FC^{t}(D)$ we have
  \[
\dist_{\FC^{t}(D)} (F, F') \ge \frac{1}{2} \dist_{\Delta_{D}(D^-) \times \Delta_{D}(D^+)} (
\rho(F), \rho(F')).
    \]
  \end{prop}

\subsection{The facet-vertex map} \label{sec:fv}
Consider the composition of the
pair embedding with the projection
\[
\FC^{t}(D) \to \Delta(D^-) \times \Delta(D^+) \onto \Delta(D^-).
\]
Given a face $F \in \Delta(D^-)$ it is difficult in general to decide
whether it is in the image of this map, and if it is, in how many
ways. (In the language of the previous section, deciding membership of $\Delta_D(D^+)$ and
$\Delta_D(D^-)$ is complicated in general.)

However the situation simplifies significantly when $F$ is a facet of
$D^-$:

\begin{prop}
  For any facet $F$ of $D^-$ there exists a unique vertex $x_F$ of
  $D^+$ such that $(F,\{x_F\})$ is in the image of $\rho$.
\end{prop}

\begin{proof}
We first claim that any $F$ facet of $D^-$ is a ridge of the whole
drum $D$. Indeed, take any functional $f$ which is zero on $F$ and
$>0$ on $D^-$. If $g$ is a functional which is zero on $D^-$ and
$=1$ on $D^+$ then, for $N$ large enough, $f + Ng$ is $\ge 0$ on $D$
and zero exactly on $F$, which proves that $F$ is a ridge.

In any polytope there are exactly two facets incident to any ridge. In
the case of our ridge $F$, one such facet is $D^-$. The other facet
(call it $A$) intersects $D^+$ in a single point $x$ (by our simpliciality
assumption \eqref{eq:simplicial_assumption}). This proves
existence. To get \change{uniqueness} notice that any facet of $D$ which
intersects $D^-$ in $F$ is necessarily incident to $F$, and hence is
equal to $A$, by the statement opening this paragraph.
\end{proof}

We call the resulting maps
\begin{align*}
\phi^- : \text{facets of $D^-$} &\to \text{vertices of $D^+$} \\
\phi^+ : \text{facets of $D^+$} &\to \text{vertices of $D^-$}
\end{align*}
the \emph{facet-vertex maps}.

\begin{remark}
  These maps (in a slightly different notation) play a key role in the 1970s counter-example to Hirsch in
  the unbounded setting \cite{KleeWalkup}.
\end{remark}

\begin{remark}
  One can use the facet-vertex maps to construct a bipartite graph
  with nodes given by the facets of $D^+$ and $D^-$, and an edge from
\change{  $F \in D^+$ to $F' \in D^-$ (resp. $F'$ to $F$) if
  $\phi^+(F) \in F'$ (resp. $\phi^-(F') \in F$)}. It is then easy to see from
  Proposition \ref{prop:bound_width} that $D$ is non-Hirsch (i.e. of
  width $\ge d+1$) if and only if the resulting graph contains no
  oriented 2-cycles. This was first observed in
  \cite[Proposition 2.3]{MSW} (the authors call the above graph the
  ``incidence pattern''). In our opinion, this observation provides
  the easiest way to check  that Santos' example is of width $\ge 6$.
\end{remark}

\subsection{Determining the facet-vertex map} \label{sec:method}

There is a beautiful way to determine the facet-vertex maps. This is
the method that we will use in practice below.

In order to describe this method, let us ignore the last coordinates and regard the top and bottom
drum skins as embedded in the same affine space. Thus, if our drum
$D$ sits inside $\RM^d$, we regard $D^+$ and $D^-$ as sitting inside
the same $\RM^{d-1}$, via the affine embeddings
\begin{equation} \label{eq:embed}
\RM^{d-1} \times \{ 1 \} = \RM^{d-1} = \RM^{d-1} \times \{-1\}
  \end{equation}
(In other words, we ignore the last (``drum'') coordinate).

Now, for any facet $F$ of $D^+$ we can consider a defining
functional $f_F$. (Recall that this means that $f_F$ is $\ge 0$ on
$D^+$ and is $=0$ on $F$.) We claim:

\begin{lem}\label{lem:minimum} We have
\[
\phi^+(F) = v
\]
where $v$ is the vertex of $D^-$ where $f_F$ obtains its minimum value.
\end{lem}

\begin{proof} Consider the unique facet $F' \ne D^+$ of $D$ which contains $F$, and
  let $f$ denote its defining functional. By definition, $F'$ has a
  unique vertex not belonging to $D^+$ and $\phi^+(F) = v$.  Let us write
  $f$ out in coordinates
\[
f_+ : (x_1, \dots, x_d) \mapsto a_0 + \sum_{i=1}^d a_ix_d.
\]

If we restrict $f$ to $\RM \times \{ 1 \}$ we get the functional
  \begin{equation} \label{eq:new func plus}
f_+ : (x_1, \dots, x_{d-1}) \mapsto a_0 + a_d +  \sum_{i=1}^{d-1} a_ix_d.
\end{equation}
Because $f$ is zero on $F'$ and $\ge 0$ on $D$, $f_+$ is zero on $F$ and $>
0$ on $D^+ - F$. In other words, $f_+$ is a defining functional for $F
\subset D^+$. Thus $f_+$ and $f_F$ agree, up to positive scalar multiple.

  If we restrict $f$ to $\RM \times \{ -1 \}$ we get the functional
  \begin{equation} \label{eq:new func}
f_- : (x_1, \dots, x_{d-1}) \mapsto a_0 - a_d +  \sum_{i=1}^{d-1} a_ix_d
    \end{equation}
  which is $\ge 0$ on
$D^-$ and obtains its unique minimum value on $D^-$ at $v$.\footnote{This
minimum value is of course $0$, but we won't need this.} Hence $f +
\lambda$ also obtains its minimum value on $D^-$ at $v$, for any $\lambda \in \RM$.

In particular, $f_+ = f_- + 2a_d$ obtains its minimum at $v$ and the
lemma is proved.\end{proof}

A more geometric way of formulating this is as follows. Consider first
$D^+$ and the fixed facet $F$. We can consider the family of parallel
hyperplanes $H_t$ given by $f_F = t$ for $t \in \RM$. For large values
of $t$, $H_t$ does not intersect our polytope. As we
increase $t$ we get a family of parallel hyperplanes that 
eventually intersects $D^+$ for small positive values of $t$. At $t = 0$, $H_t$ intersects $D^+$ in
$F$, and then $H_t$ does not intersect $D^+$ for any negative
values of $t$: 
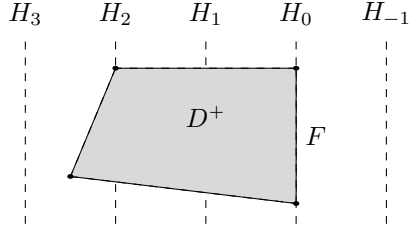
\begin{equation*}\begin{array}{c}
\begin{tikzpicture}[xscale=1.2,yscale=.9]

\coordinate (A) at (-1.5, -.6);
\coordinate (B) at (1, -1);
\coordinate (C) at (1, 1);
\coordinate (D) at (-1, 1);
\coordinate (E) at (-0.5, 0.5);

\foreach \i in {-2, -1, 0, 1, 2} {
    \draw[dashed] ($(\i,-1.3)$) -- ($(\i,1.5)$);
}

\draw[fill=gray!30] (A) -- (B) -- (C) -- (D) -- cycle;

\node at (-2, 1.8) {$H_{3}$};
\node at (-1, 1.8) {$H_{2}$};
\node at (0, 1.8) {$H_{1}$};
\node at (1, 1.8) {$H_{0}$};
\node at (2, 1.8) {$H_{-1}$};

\node[above] at (0,0) {$D^+$};
\node[right] at (1,0) {$F$};


\draw[dashed] (A) -- (B);
\draw[dashed] (B) -- (C);
\draw[dashed] (C) -- (D);
\draw[dashed] (D) -- (A);

\fill (A) circle (1pt);
\fill (B) circle (1pt);
\fill (C) circle (1pt);
\fill (D) circle (1pt);
\end{tikzpicture}\end{array}\end{equation*}

Now let us keep our family of hyperplanes $H_t$, and instead focus on
$D^-$. For very negative values of $t$, $H_t$ does not intersect $D^-$. As we increase $t$ we get a family of hyperplanes
getting closer and closer to $D^-$. Our simplicial assumption
guarantees that the first time the hyperplanes $H^t$ intersect $D^-$,
the intersection will be in a single vertex $\{ v \}$:
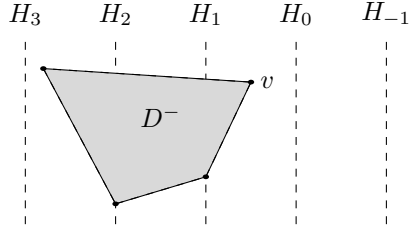
\begin{equation*}\begin{array}{c}
                   \begin{tikzpicture}[xscale=1.2,yscale=.9]
                     
\coordinate (A) at (.5, .8);
\coordinate (B) at (-1.8, 1);
\coordinate (C) at (-1, -1);
\coordinate (D) at (0, -.6);

\foreach \i in {-2, -1, 0, 1, 2} {
    \draw[dashed] ($(\i,-1.3)$) -- ($(\i,1.5)$);
}

\draw[fill=gray!30] (A) -- (B) -- (C) -- (D) -- cycle;

\node at (-2, 1.8) {$H_{3}$};
\node at (-1, 1.8) {$H_{2}$};
\node at (0, 1.8) {$H_{1}$};
\node at (1, 1.8) {$H_{0}$};
\node at (2, 1.8) {$H_{-1}$};

\node[above] at (-.5,0) {$D^-$};
\node[right] at (A) {$v$};


\draw[dashed] (A) -- (B);
\draw[dashed] (B) -- (C);
\draw[dashed] (C) -- (D);
\draw[dashed] (D) -- (A);

\fill (A) circle (1pt);
\fill (B) circle (1pt);
\fill (C) circle (1pt);
\fill (D) circle (1pt);
\end{tikzpicture}\end{array}\end{equation*}
Then $v$ is the image of our facet-vertex map:
\[
\phi^+(F) = v.
  \]

  \begin{remark}
    Another way of thinking about this is the following: All facets in
    $D^+$ determine linear 
    programming problems on $D^-$ whose solutions determine the
    facet-vertex map in one direction. Similarly, all facets of $D^-$
    determine linear programming problems on $D^+$ whose solution
    determine the facet-vertex map in the other direction. This
    formulation of the two facet-vertex maps as ``paired linear
    programs'' seems attractive.
  \end{remark}

  \subsection{Determining the pair embedding} \label{sec:method pair}
As in the previous section, we regard $D^+$ and $D^-$ as sitting
inside the same $\RM^{d-1}$ via the identifications (of affine spaces)
\begin{equation} \label{eq:id7}
\RM^{d-1} \times \{ 1 \} = \RM^{d-1} = \RM^{d-1} \times \{ -1 \}.
\end{equation}

Suppose given a facet $F^-$ of $D^-$ of dimension
$d^-$ and a facet $F^+$ of $D^+$ of dimension $d^+$. We are interested
in determining whether $(F^-, F^+)$ is in the
image of the pair embedding, and thus corresponds to a facet $F$ of
$D$. We have already seen that for this to be the case we must have
\begin{equation} \label{eq:dim bound}
d^+ + d^- = d-2.
\end{equation}
The following provides an answer:

\begin{lem} \label{lem:pair prog} 
  A pair $(F^-, F^+)$ as above occurs in the image of the
  pair embedding if and only if there exists a functional $f$ on
  $\RM^{d-1}$ whose minima on $D^-$ (resp. $D^+)$ is equal to $F^-$
  (resp. $F^+$).
\end{lem}

\begin{proof}
First suppose that such a functional $f$ exists. The space of all
functionals on $\RM^d$ which agree with $f$ on $\RM^{d-1} \times \{ -1
\}$ is of dimension $1$, and within this space we take the unique
functional $\widetilde{f}$ which has the same value on $F^-$ and $F^+$. By adding a
constant, we can assume that $\widetilde{f}$ takes the value zero on
$F^-$ and $F^+$. Now $\widetilde{f}$ is $\ge 0$ on $D^-$ and $D^+$ and
hence is $\ge 0$ on the whole drum $D$. The intersection of $D$ with
the set $\widetilde{f} = 0$ is a facet $F$ containing $F^-$ and $F^+$, by
\eqref{eq:dim bound}. Thus $(F^-, F^+) = (F_\bot, F_\top)$ as claimed.

On the other hand, consider a facet $F$ of $D$ which is not equal to
$D^-$ or $D^+$, and let $F^- = F \cap D^-$ and $F^+ = F \cap
D^+$. Let $f_D$ be a defining functional for $F$ and let $f^-$ (resp. $f^+)$ denote its restriction to $\RM^{d-1}
\times \{ -1\}$ (resp. $\RM^{d-1} \times \{ 1 \})$. Under the identifications
\eqref{eq:id7}, we have $f^+ = f^- + \gamma$ for some constant
$\gamma$. Thus the minimum of $f^-$ on $D^-$ (resp. $D^+)$
consists of $F^-$ (resp. $F^+$). In particular, $f^-$ is our desired
functional.
\end{proof}

This lemma has the following corollary, which will prove useful
below:

\begin{cor} \label{eq:ridges}
  Consider two facets $F_1, F_2$ of $D^+$ which intersect in
  a ridge $R$. If $F_1$ and $F_2$ map to the same vertex of $D^-$ under the
  facet-vertex map, then the ridge $R$ is not in the image of the pair embedding.
\end{cor}

\begin{proof}
  Consider defining functionals $f_1$ and $f_2$ for $F_1$ and
  $F_2$. Any positive linear combination $\alpha f_1 + \beta f_2$ with
  $\alpha, \beta > 0$ defines a functional vanishing on $R$, and all
  functionals vanishing on $R$ and $\ge 0$ on $D^+$ are of this
  form. Now, observe that the minimum of $\alpha f_1 + \beta f_2$ on
  $D^-$ is always $v$ for any $\alpha, \beta > 0$. Thus, there is no
  functional $f$ which is zero on $R$, $\ge 0$ on $D^+$ and has
  minimum on $D^-$ consisting of anything but $\{ v \}$. We conclude
  by the previous lemma.
\end{proof}

\section{The $+k$ family}\label{sec:construction}

In this section we introduce our $+k$ family of drums, and study
their geometry. First we give the definition, then we spend
considerable time getting an explicit description of the top and
bottom drum skins. We then compute the facet-vertex map, and hence
obtain some basic knowledge of the pair embedding. We then use this to
establish that they are of width $\ge 5 + k$.

\subsection{Definition} Fix $k \ge 1$. We consider the following points in $\RM^5$:
\begin{align*}
  m_{\pm,\pm} &= (0,0,\pm 3,\pm 3,1) \\
  p_{\pm} &= (98,0,\pm 1,0,1) \\
  a_1 &= (100, 0, 0, 0, 1) \\
  a_2 &= (x_2,y_2,0,0,1) \\
  \vdots & \qquad \vdots \\
     a_{k} &= (x_{k},y_{k},0,0,1) \\       
  a_{k+1} &= (75,75,0,0,1) 
\end{align*}
We denote the set of these points $\{ m_{\pm\pm}, p_\pm,
  a_i \}$. 
  \begin{remark}
    \begin{enumerate}
    \item We will be more precise about the $a_i$ in a moment.
    \item Really $m_{\pm\pm}$ is shorthand for 4 points. For example,
      \[ m_{+-} = (0,0,+3,-3,1). \] Similarly, $p_{\pm}$ is shorthand
      for 2 points, for example
      \[p_+ = (98,0,1,0,1). \]
    \end{enumerate}
  \end{remark}

In order to define our drum, we use a group of symmetries to enlarge our
set of points. Let
\[
\e_i = \text{sign change at $i^{th}$ coordinate}.
\]
So for example, $\e_2 a_{k+1} = (75, -75, 0, 0, 1)$. Let $\tau'$
denote the permutation of the coordinates induced by
$(3,4,2,1,5)$. Thus, for example, $\tau' p_- = (0,-1,98,0,1)$. Now define
\[
\tau = \e_5 \tau'.
\]
Thus, $\tau$ permutes the first 4 coordinates and flips the sign of
the last coordinate. \change{Note that $\tau^2$ is the permutation of
  the coordinates induced by $(2,1,4,3)$.} Finally, consider the group:
\[
\Gamma = \langle \tau , \e_i \; | \; 1 \le i \le 4 \rangle.
\]
\change{Note that $\langle \e_i \; | \; 1 \le i \le 4 \rangle$ is a
  normal subgroup of $\Gamma$ with quotient isomorphic to $\ZM/4\ZM$. In
  particular
  \begin{equation}\label{orderofgamma}
    |\Gamma| = 64.
    \end{equation}.}

We now define:
\[
D_k = \text{convex hull of } \{ g(v) \; | \;  v \in \{ m_{\pm\pm}, p_\pm,
  a_i \} \text{ and } g \in \Gamma \}
  \]
By construction, $D_k$ is invariant under $\Gamma$. We will use this
extensively below to simplify computations.

Also note that $\tau$ exchanges top and bottom drum skins, and thus
its square
\[
\sigma = \tau^2
\]
preserves the top and bottom drum skins. In particular, top and bottom
drum skins are invariant under the subgroup
\[
\Gamma^+ = \langle \sigma, \e_i \; | \; 1 \le i \le 4 \rangle \subset \Gamma.
\]
In fact, it is easy to see that $\Gamma^+$ is precisely the subgroup of
$\Gamma$ which stabilizes the top and bottom drum skins setwise.

\begin{remark}
  The idea to consider drums which are symmetric under sign flips and
  a permutation is due to Santos \cite{Santos}. His counter-example to
  the Hirsch conjecture is of this form.  Note that, up to permutation of any of the first four coordinates
  coordinates we could replace $(3,4,2,1)$ by any conjugate in
  $S_4$. Thus the important feature of $\tau$ is that it is an
  $4$-cycle.
\end{remark}

\begin{remark}
  Our permutation is determined by \verb|MAIN_SYMMETRY| in the colab \cite{colab},
  up to inversion and conversion to  python indexing: $1 \mapsto 0, 2
  \mapsto 1, \dots$. The transformation $\tau$ is given by the
  function \verb|permute_axes|.
\end{remark}

\subsection{Critical assumptions} \label{sec:critical}
We now make explicit our assumptions
on the points $a_1, a_2, \dots, a_{k+1}$. In picturing the points $\{ m_{\pm\pm}, p_\pm,
a_i \}$ it is useful to ignore the last coordinate (which is always 1)
and project to the first two coordinates. When we do so, we get a
picture as in Figure \ref{fig:points}. Note that all points in this
figure are uniquely determined by their images in the projection with
the exception of the image of $m_{++}$ (which has the four points
$m_{\pm\pm}$ in its preimage) and the image of $p_+$ (which has the
two points $p_{\pm}$ its preimage).

  \begin{figure} 
    \caption{Our points $m_{\pm\pm}$, $p_{\pm}$ and
       $a_i$ for $k = 5$. (Not to scale!)}
       \label{fig:points}
 \centering 
\begin{tikzpicture}
 [acteur/.style={circle, fill=black,thick, inner sep=2pt, minimum size=0.2cm}] 
\node (m) at (-10,0) [acteur][label=$m_{\pm\pm}$]{};
\node (p) at ( -1,0)[acteur][label={$p_{\pm}$}]{}; 
\node (a1) at ( 0,0) [acteur][label=$a_1$]{}; 
\node (a6) at ( -2.5,7.5) [acteur][label=$a_6$]{}; 
\node (a5) at (106.5:6) [acteur][label={$a_5$}]{}; 
\node (a4) at (105:4) [acteur][label={$a_4$}]{}; 
\node (a3) at (103:2) [acteur][label={$a_3$}]{}; 
\node (a2) at (101:1) [acteur][label={$a_2$}]{};
\node[below] (z) at (m) {\tiny $(0,0)$};
\node[below] (a1') at (a1) {\tiny $(100,0)$};
\node[below] (a6') at (a6) {\tiny $(75.75)$};
\node[below] (p') at (p) {\tiny $(98,0)$};
\draw[gray!50!white] (a1) -- (m) -- (a6) -- (a5) -- (a4) -- (a3) -- (a2) --
(a1) -- (a1);
\node (m) at (-10,0) [acteur][label=$m_{\pm\pm}$]{};
\node (p) at ( -1,0)[acteur][label={$p_{\pm}$}]{}; 
\node (a1) at ( 0,0) [acteur][label=$a_1$]{}; 
\node (a6) at ( -2.5,7.5) [acteur][label=$a_6$]{}; 
\node (a5) at (106.5:6) [acteur][label={$a_5$}]{}; 
\node (a4) at (105:4) [acteur][label={$a_4$}]{}; 
\node (a3) at (103:2) [acteur][label={$a_3$}]{}; 
\node (a2) at (101:1) [acteur][label={$a_2$}]{};
\end{tikzpicture} 
\end{figure}
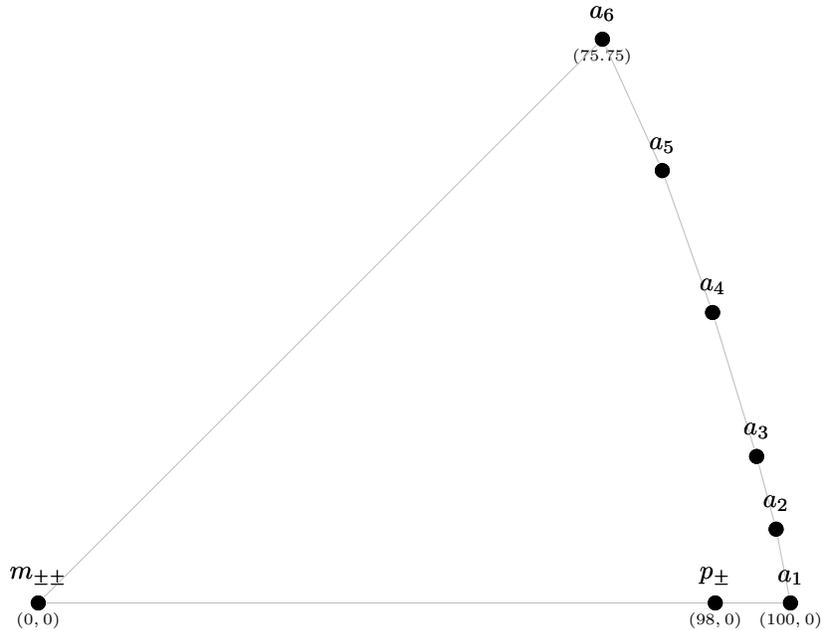

We have drawn the points in Figure \ref{fig:points} to try to
illustrate our key assumptions. \change{Our first is that out points lie in
the first quadrant below the line $y = x$:
\begin{gather} \label{eq:first quadrant assumption}
100 = x_1 > x_2 > \dots > x_k > x_{k+1} = 75 \quad \text{and} \\
0 = y_1 < y_2 < \dots y_k < y_{k+1} = 75. \label{eq:first quadrant assumption2}
\end{gather}
}
Our second is a simple convexity assumption:
\begin{gather} \label{eq:convex assump}
\begin{array}{c} \text{The points $a_1, a_2 \dots a_{k+1}$ are in
  convex position,} \\
  \text{i.e., no $a_j$ is in the convex hull of the other $a_i$.}
\end{array} \end{gather}
Our third assumption says that the points $a_i$ ``almost lie on a line":
\begin{gather}
  \label{eq:ASSUMP}
  \begin{array}{c}  \text{For all $1 \le i \le k$, the line drawn through $a_i$ and $a_{i+1}$}
  \\
  \text{meets the $x$ axis at $u$, where $100 \le u \le 102$.} \end{array}
\end{gather}

It is clear that we can always chose $k+1$ points $a_i$ satisfying \eqref{eq:ASSUMP}.


\change{The following elementary consequence of our assumptions will be important later.}

\begin{lem}
  The line drawn through $a_i$ and $a_{i+1}$ meets the $y$ axis at
  $v_i$, where
  \begin{equation}
    \label{eq:ASSUMPy}
    v_i \ge \frac{75}{1-75/102} > 283
  \end{equation}
\end{lem}

\begin{remark}
  The reader is reminded of Remark \ref{rem:constants} when
  trying to make sense of constants!
\end{remark}

\begin{proof}
  \change{Consider the lines $L_i$ drawn through $a_i$ and $a_{i+1}$
    for $1 \le i \le k$. Let $u_i$ (resp. $v_i$) denote the
    intersection point of $L_i$ with the $x$-axis (resp. $y$-axis). By
    our convexity assumption \eqref{eq:convex assump} we have:
    \begin{equation}
      v_1 > v_2 > \dots > v_k
    \end{equation}
    By our convexity assumption \eqref{eq:convex assump} and our
    ``almost on a line'' assumption \eqref{eq:ASSUMP} we have:
    \begin{equation}
      u_1 < u_2 < \dots < u_k \le 102.
    \end{equation}
It follows (again by convexity) that if we consider the line $L$
through $a_{k+1} = (75,75)$ and
$(102,0)$ it meets the $y$ axis at a point $u$ satisfying
\[ v_k > u. \]
Now an explicit computation yields that
\[
  u=\frac{75}{1-75/102} 
\]
and the lemma follows.
}
\end{proof}

\change{
  Finally, we count the vertices of $D_k$,
  justifying the count in the introduction:}

\begin{lem}
$D_k$ has $16k + 24$ vertices.
\end{lem}

\begin{proof}
  \change{The reader may easily check that our assumptions guarantee
    that any point in the $\Gamma$ orbit of $m_{\pm,\pm}$, $p_{\pm}$
    is indeed a vertex of $D_k$. We count the number of these orbits
    by exploiting the $\Gamma$-action and the orbit stabilizer
    theorem. Also note that up to the $\Gamma$-action we only need to
    consider $m_{++}$, $p_+$ and the $a_i$ for $1 \le i \le k+1$.}
      
\change{ The stabilizer of $m_{++}$ is $\langle \e_1, \e_2, \tau^2
    \rangle$, and hence its orbit has size $8$. Similarly the orbit of $a_{k+1}$ has size
    $8$. For $a_1$ the stabilizer is $\langle \e_2, \e_3, \e_4 \rangle$ and so
    its orbit also has size 8. The stabilizer of $p_+$ is $\langle
    \e_2, \e_4 \rangle$ and hence its orbit is of size 16.
    Finally, for $a_i = (x_i, y_i, 0, 0, 1)$ for $2 \le i \le k$ our
    assumptions \eqref{eq:first quadrant assumption} and
    \eqref{eq:first quadrant assumption2} guarantee that $x_i, y_i > 0$ and $x_i \ne y_i$;
    hence its stabilizer is $\langle \e_3, \e_4 \rangle$ and its orbit
    is of size $16$.}

    \change{Adding up we get
    \[
      8 + 8 + 8 + 16 + (k-1)16 = 16k + 24
\]
as claimed.
  }
  \end{proof}

\subsection{The top drum skin} It seems hopeless to study $D_k$ in
detail. However we can get a detailed understanding of the top drum
skin $D_k^+$, which is what we do here.

By abuse of notation, we use the same notation as earlier for certain
vertices of $D_k^+$:
\begin{align*}
  m_{\pm,\pm} &= (0,0,\pm 3,\pm 3) \\
  p_{\pm} &= (98,0,\pm 1,0) \\
  a_1 &= (100, 0, 0, 0) \\
  a_2 &= (x_2,y_2,0,0) \\
  \vdots & \qquad \vdots \\
     a_{k} &= (x_{k-1},y_{k-1},0,0) \\       
  a_{k+1} &= (75,75,0,0) 
\end{align*}
One checks easily that all vertices of $D_k^+$ are obtained from these
vertices by acting by the group $\Gamma^+$ generated by the
permutation $(12)(34)$ of coordinates together with all sign flips.

\begin{remark}
As described in the introduction, the motivation for this construction
was Santos' original counterexample, which was the closure of the five
points
\[
(18,0,0,0,1), (0,0,45,0,1), (15,15,0,0,1), (0,0,30,30,1), (10,0,0,40,1) \]
under $\Gamma$. We started with this construction and experimented with adding points to the fundamental domain so that the width increases, until we eventually found a pattern that generalizes.   
\end{remark}

\begin{remark} \label{rem:double pancake}
  One aspect of the geometry of $D$ is worth pointing
  out, namely its curious ``double pancake'' structure. Consider a
  pancake embedded in $\RM^3$ so that it is symmetric 
  about the origin in the $x$ and $y$ coordinates. Points on the edge
  of the pancake have large $x$ and $y$ coordinates, and all points of
  the pancake have small $z$ coordinates. One can imagine a $4d$
  version of this, (a ``$4d$ pancake'') with two large coordinates, and two small
  coordinates.

  It is remarkable that in our example the top and bottom
  drum skins are both such $4d$ pancakes: $D^+$ has large first
  two coordinates (on the order of $100$) and small second two
  coordinates (on the order of $1$). The large and small coordinates
  are swapped in $D^-$. The reader wishing to see this structure more
  explicity is referred to the cells ``Understanding the top drum
  skin'' and ``The pancake structure'' in the colab \cite{colab}.

  It is also interesting to note that amongst our symmetric searches,
  those drums which exhibited the ``double pancake'' structure where
  much more likely to provide drums of large width. \change{(See \cite[Lemma
  2.13]{MSW} and Remark \ref{rem:doublethm}.)}
\end{remark}

\subsection{Facets in the top drum skin} Recall that our first goal is
to come to grips with the top drum skin. The following is the first
significant step towards the proof of our main theorem:
\begin{prop} \label{prop:top_skin}
  Every facet of $D_k^+$ is conjugate under $\Gamma^+$ to one of the
  following facets (for $1 \le i \le k$):
  \begin{align*}
    B &= \{ m_{++}, m_{+-}, p_+, a_{k+1} \}, \\
    E_i &= \{ m_{++},m_{-+},a_i,a_{i+1} \}, \\
    C_i & = \{ m_{++},p_+, a_i, a_{i+1} \}.
  \end{align*}
\end{prop}

\begin{remark}
  The reader is referred to Figures \ref{fig:facet B}, \ref{fig:facet E_i} and
  \ref{fig:facet C_i} which depict the projections of these 3 types of facets. 
\end{remark}


  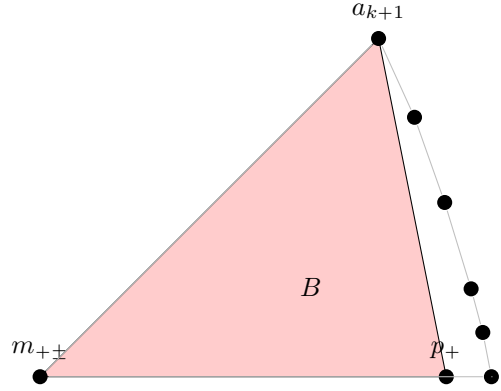
\begin{figure}[p] 
    \caption{Projection of the facet $B$.}
    \label{fig:facet B}
 \centering 
\begin{tikzpicture}
 [scale=.6,acteur/.style={circle, fill=black,thick, inner sep=2pt, minimum size=0.2cm}] 
%
%
\filldraw[draw=black,fill=red!20] (-10,0) -- (-1,0) -- ( -2.5,7.5) -- cycle;
\node (m) at (-10,0) [acteur][label=$m_{+\pm}$]{};
\node (p) at ( -1,0)[acteur][label={$p_{+}$}]{}; 
\node (a1) at ( 0,0) [acteur]{};
\node (a6) at ( -2.5,7.5) [acteur][label=$a_{k+1}$]{}; 
\node (a5) at (106.5:6) [acteur]{};
\node (a4) at (105:4) [acteur]{};
\node (a3) at (103:2) [acteur]{};
\node (a2) at (101:1) [acteur]{};
\draw[gray!50!white] (a1) -- (m) -- (a6) -- (a5) -- (a4) -- (a3) -- (a2) --
(a1) -- (a1);
\node at (-4,2) {$B$};
\end{tikzpicture} 

\end{figure}

  \begin{figure}[p]
    \caption{Projection of the facet $E_i$.}
     \label{fig:facet E_i}
 \centering 
\begin{tikzpicture}
 [scale=.6,acteur/.style={circle, fill=black,thick, inner sep=2pt, minimum size=0.2cm}] 
%
%
\filldraw[draw=black,fill=red!20] (-10,0) -- (103:2) -- (105:4) -- cycle;
\node (m) at (-10,0) [acteur][label=$m_{\pm+}$]{};
\node (p) at ( -1,0)[acteur]{};
\node (a1) at ( 0,0) [acteur]{};
\node (a6) at ( -2.5,7.5) [acteur]{};
\node (a5) at (106.5:6) [acteur]{};
\node (a4) at (105:4) [acteur][label={$a_{i+1}$}]{}; 
\node (a3) at (103:2) [acteur][label={$a_i$}]{}; 
\node (a2) at (101:1) [acteur]{};
\draw[gray!50!white] (a1) -- (m) -- (a6) -- (a5) -- (a4) -- (a3) -- (a2) --
(a1) -- (a1);
\node at (-4,2) {$E_i$};
\end{tikzpicture} 
\end{figure}

  \begin{figure}[p] 
    \caption{Projection of the facet $C_i$.}
    \label{fig:facet C_i}
 \centering 
\begin{tikzpicture}
 [scale=.6,acteur/.style={circle, fill=black,thick, inner sep=2pt, minimum size=0.2cm}] 
%
%
\filldraw[draw=black,fill=red!20] (-10,0) -- (-1,0) -- (103:2) --
(105:4) -- cycle;
\draw (-10,0) -- (103:2);
\draw (-1,0) -- (105:4);
\node (m) at (-10,0) [acteur][label=$m_{++}$]{};
\node (p) at ( -1,0)[acteur][label={$p_{+}$}]{}; 
\node (a1) at ( 0,0) [acteur]{};
\node (a6) at ( -2.5,7.5) [acteur]{};
\node (a5) at (106.5:6) [acteur]{};
\node (a4) at (105:4) [acteur][label={$a_{i+1}$}]{}; 
\node (a3) at (103:2) [acteur][label={$a_i$}]{}; 
\node (a2) at (101:1) [acteur]{};
\draw[gray!50!white] (a1) -- (m) -- (a6) -- (a5) -- (a4) -- (a3) -- (a2) --
(a1) -- (a1);
\node at (-3,.7) {$C_i$};
\end{tikzpicture} 
\end{figure}

  Our proof is in two steps:
\begin{enumerate}
\item We prove that the listed facets are indeed facets. To
  do this, we produce a functional which is zero on the
  specified vertices, and $>0$ on all other vertices of $D_k^+$.
\item We prove that we have all facets on our list up to $\Gamma^+$
  conjugacy. For this we check that all facets incident to a facet on
  our list are $\Gamma^+$-conjugates of facets on our list.
\end{enumerate}
Actually, the second step gives us slightly more, namely good control
of the incidence between the facets in the top drum skin. We exploit
this fact in the next section.

\begin{proof}[Proof that the claimed facets are facets:]
Now, we roll up our sleeves. We handle each family of facets separately:

  \emph{Facet $B$}: Consider the functional given by the row vector:
  \[f_B =  (-1, -\frac{24}{25}, -49, 0, 3 \cdot 49) \]
One checks easily that $f_B$ vanishes at $m_{+,\pm}$, $p_+$ and
$a_{k+1}$. Furthermore, \begin{enumerate}
\item $f_B$ is $>0$ on $m_{-,+},m_{-,-}$.
  \item $f_B$ is $>0$ on all points of the $\Gamma^+$ orbit of $p_+$ except $p_+$.
  \end{enumerate}
In order to establish that $f_B$ is $>0$ on the $\Gamma^+$ orbits
of $a_1, a_2, \dots, a_k$ we use a convexity
argument. Consider the restriction of $f_B$ to the plane
$(*,*,0,0)$. It is given by the functional
\[
-x - \frac{24}{25}y + 3 \cdot 49.
\]
It is now easy to check that this functional is $\ge 0$ on the convex
hull of the $\Gamma^+$-orbits $a_1, a_2, \dots, a_{k+1}$ and has a
unique zero at $a_{k+1}$. Hence $B$ is a facet as claimed.

\emph{Facet $E_i$}. Let us view $\RM^4$ as the direct sum of vector spaces $V_1
\oplus V_2$, where $V_1$ (resp. $V_2$) is the subspace corresponding
to the first and second (resp. third and fourth) coordinates. Note
that the $a_i$ lie in $V_1$ and the $m_{\pm \pm}$ lie in 
$V_2$. Consider \emph{linear} functionals on $V_1$ and $V_2$ respectively
defined as follows:
\begin{enumerate}
\item $f_1$ is $-1$ on $a_i$ and $a_{i+1}$ (and hence $>-1$ at all
  other points in the $\Gamma^+$-orbit of the $a_i$'s);
\item $f_2$ is $-1$ on $m_{++}$ and $m_{-+}$ (and hence $>-1$ on $m_{-+}$ and
  $m_{--}$).
\end{enumerate}
Now consider the functional $f = f_1 + f_2 + 1$ defined on $\RM^4$. By
property (1), $f$ is $\ge 0$ on the $\Gamma^+$-orbit of the $a_i$ and
$=0$ only at $a_i$ and $a_{i+1}$. By property (2), it is zero on $m_{++}$
and $m_{-+}$ and $>0$ on $m_{-+}$ and   $m_{--}$.

It remains to check that $f$ is $>0$ on the $\Gamma^+$-orbit of
$p_+$. Firstly, note that $f_2((c,d)) = -\frac{1}{3}d$ and hence
$f((\pm 98, 0, \pm 1, 0)) > 0$. Also, amongst points of the form $(0,
\pm 98, 0, \pm 1)$, $f$ achieves its minimum at $(0,98,0,1)$ by
inspection. Hence we are done if we can show that
\begin{equation}
  \label{eq:ineqtarget}
f_1((0,98)) + f_2((0,1)) > -1.  
\end{equation}
By our formula for $f_2$, we have $f_2((0,1)) = -\frac{1}{3}$. Now,
consider the line through $a_i$ and $a_{i+1}$ as in \S
\ref{sec:critical}. We know that $f$ takes the value $-1$ along this
line, and hence, denoting by $v$ the intersection point of this line
with the $y$-axis, $f_1((0,v)) = -1$. Hence
\[
f_1((0,98)) = 98 \cdot f_1((0,1)) = \frac{98}{v} f_1((0,v)) = -
\frac{98}{v} > -\frac{98}{283}.
\]
where we have used \eqref{eq:ASSUMPy} for the last inequality. Hence
we have \eqref{eq:ineqtarget} and $E_i$ is a facet as claimed.

\emph{Facet $C_i$:} The points $m_{++}, p_+, a_i, a_{i+1}$ are clearly
affinely independent, and hence there exists a functional $f$ which is $=0$
on these four points, and $> 0 $ at 0.

More precisely, $f$ is defined up to a positive scalar. Hence,
altering $f$ if necessary, we may
assume that $f = f' + 1$ where $f'$ is linear. Let us further decompose $f$ as $f = f_1 + f_2
+ 1$ where $f_1$ (resp. $f_2$) is a linear functional on the first two
(resp. second two) coordinates in the decomposition $\RM^4 = \RM^2
\oplus \RM^2$ that we used above. We know that:
\begin{enumerate}
\item $f_1$ takes the value $-1$ at $a_i$ and $a_{i+1}$ (because $f$
  vanishes at $a_i$, $a_{i+1}$);
\item $f_2$ takes the value $-1$ at $(3,3)$ (because $f$ vanishes at
  $m_{++} = (0,0,3,3)$).
\end{enumerate}
We still have to use the fact that $f$ vanishes at $p_+$. In order to do so set
\[
\delta_1 = f_1((98,0)) \quad \text{and} \quad \delta_2 = f_2((1,0)).
\]
Because $f(p_+) = f_1((98,0))  + f_2((1,0)) + 1 = 0$ we have
\begin{equation}
  \label{eq:d1d2}
  \delta_1 + \delta_2 = -1.
\end{equation}

Now let $u$ be defined such that $f_1((u,0)) = -1$. (In other words,
$u$ is the intersection point of the line through $a_i$ and
$a_{i+1}$ and the $x$-axis, considered in \S \ref{sec:critical}.) Because
\[ -1 = f_1((u,0)) = u f_1((1,0)) = \frac{u}{98} \delta_1\]
we arrive at
\begin{equation} \label{eq:d1}
\delta_1 = -\frac{98}{u}.
\end{equation}
Note that our fundamental assumption \eqref{eq:ASSUMP} yields that
$100 \le u \le 102$ and hence that
\[
 -0.98 \le \delta_1 < -0.96
\]
(We have approximated by real numbers which will give us bounds
sufficient for our needs. There is some wiggle room here.) In
particular, using \eqref{eq:d1d2} we deduce that
\[
-0.04 < \delta_2 \le -0.02 
\]
  
Finally, if we write $f_2 = (\alpha, \beta)$ then $f_2((3,3)) = -1$
gives the equation
\begin{equation} \label{eq:alphabeta}
\alpha + \beta = -\frac{1}{3}
\end{equation}
and $f_2((1,0)) = \delta_2$ yields $\alpha = \delta_2$. Hence
\begin{equation}
  \label{eq:f2}
  f_2 = (\delta_2, -\frac{1}{3}- \delta_2).
\end{equation}

Thus, $f_2$ achieves its minimum ($=-1$) on the set $\{m_{++}, \ldots,
m_{--}\}$ at $m_{++} = (3,3)$. This shows
that $f$ is $> 0$ on $m_{-+}, m_{+-}, m_{--}$. (Alternatively, this
follows directly from (\ref{eq:f2}).) Similarly, using the
definition of $f_1$, we see that $f$ is  $\ge 0$ on the entire
$\Gamma^+$ orbit of the $\{ a_j\}$, and is $=0$ only at $a_i$ and $a_{i+1}$.

It remains to see that $f$ is $\ge 0$ on the $\Gamma^+$-orbit of
$p_+$. We have already seen that $f$ is zero on $p_+$ and it is easy
to see that it is then necessarily $> 0$ on $(98, 0, -1,0)$ and $(-98,
0, \pm 1,0)$. It is also immediate that among points of the form
\[
(0, \pm 98, 0, \pm 1)
\]
$f$ takes its minimal value at $\sigma(p_+) = (0, 98, 0, 1)$. Thus we
are done if we can establish that
\begin{equation}
  \label{eq:target2}
  f_1((0,98)) + f_2((0,1)) > -1.
\end{equation}
By \eqref{eq:f2} we have $f_2((0,1)) = -\frac{1}{3}- \delta_2$ and
this quantity is $> -1/3$. Also, writing $v$ for the intersection
point of the line joining $a_i$ and $a_{i+1}$ with the $y$-axis (as in
\S\ref{sec:critical}) we deduce that $f_1((0,98)) > -98/283$ using
\eqref{eq:ASSUMPy}. Thus we have established \eqref{eq:target2} and
$C_i$ is a facet as claimed.
\end{proof}

\begin{proof}[Proof that we have all facets:]
We now proceed to the second half of the proof, namely proving that
have we
written down all facets of $D_k^+$. Note that all facets of $D_k^+$ are
3-simplices, and hence have four neighbours in the facet ridge graph. It is enough to check
that all four neighbours of each facet on our list are
$\Gamma^+$-conjugates of facets on our list.

Consider the following diagram, which shows ``obvious'' incidences
between facets:
\begin{equation} \label{eq:incidence}
  \begin{array}{c}
\begin{tikzpicture}
  \node (ee) at (0,4) {$E_k$};
  \node (ee-1) at (0,3) {$E_{k-1}$};
  \node (edots) at (0,2) {$\vdots$};
  \node (e2) at (0,1) {$E_{2}$};
  \node (e1) at (0,0) {$E_{1}$};
    \node (fe) at (2,4) {$C_k$};
  \node (fe-1) at (2,3) {$C_{k-1}$};
  \node (fdots) at (2,2) {$\vdots$};
  \node (f2) at (2,1) {$C_{2}$};
  \node (f1) at (2,0) {$C_{1}$};
  \node (b) at (2,5) {$B$};
  \draw (ee) -- (fe);
  \draw (ee-1) -- (fe-1);
  \draw (e2) -- (f2);
  \draw (e1) -- (f1);
  \draw (b) -- (fe);
  \draw(ee) -- (ee-1);
  \draw(ee-1) -- (edots);
  \draw(edots) -- (e2);
  \draw(e2) -- (e1);
    \draw(fe) -- (fe-1);
  \draw(fe-1) -- (fdots);
  \draw(fdots) -- (f2);
  \draw(f2) -- (f1);
\end{tikzpicture} \end{array}
\end{equation}
For example, $E_i$ and $E_{i+1}$ are incident because they share the
three vertices $m_{++}, m_{+-}$ and $a_{i+1}$. Similarly, $E_i$ and
$E_{i+1}$ share the three vertices $m_{++}, a_i$ and $a_{i+1}$.

We know that every facet is incident to 4 other facets (as $D_k^+$ is
4-dimensional). Thus it is enough to complete this graph by showing
that every facet above is incident to 4 facets, all of which are in
the $\Gamma^+$-orbit of one of the facets above. (In other words, it
remains to complete the above graph by giving the remaining
``non-obvious'' incidences, involving $\Gamma^+$-conjugates).

It is a straightforward check to see that around all facets $\ne B$ we
get the following graph:
\begin{equation} \label{eq:incidence}
  \begin{array}{c}
\begin{tikzpicture}
  \node (ee) at (0,4) {$E_k$};
  \node (ee-1) at (0,3) {$E_{k-1}$};
  \node (edots) at (0,2) {$\vdots$};
  \node (e2) at (0,1) {$E_{2}$};
  \node (e1) at (0,0) {$E_{1}$};
    \node (fe) at (2,4) {$C_k$};
  \node (fe-1) at (2,3) {$C_{k-1}$};
  \node (fdots) at (2,2) {$\vdots$};
  \node (f2) at (2,1) {$C_{2}$};
  \node (f1) at (2,0) {$C_{1}$};
  \node (b) at (2,5) {$B$};
  \node (sb) at (0,5) {$\sigma B$};
  \draw (ee) -- (fe);
  \draw (ee-1) -- (fe-1);
  \draw (e2) -- (f2);
  \draw (e1) -- (f1);
  \draw (b) -- (fe);
  \draw (sb) -- (ee);
  \draw(ee) -- (ee-1);
  \draw(ee-1) -- (edots);
  \draw(edots) -- (e2);
  \draw(e2) -- (e1);
    \draw(fe) -- (fe-1);
  \draw(fe-1) -- (fdots);
  \draw(fdots) -- (f2);
  \draw(f2) -- (f1);
  \node (epse) at (-2,4) {$\epsilon_3C_k$};
  \node (epse-1) at (-2,3) {$\epsilon_3C_{k-1}$};
  \node (epse2) at (-2,1) {$\epsilon_3C_2$}; 
  \node (epse1) at (-2,0) {$\epsilon_3C_1$};
    \node (epsce) at (4,4) {$\epsilon_4C_k$};
  \node (epsce-1) at (4,3) {$\epsilon_4C_{k-1}$};
  \node (epsc2) at (4,1) {$\epsilon_4C_2$}; 
  \node (epsc1) at (4,0) {$\epsilon_4C_1$};
    \node (bote) at (0,-1) {$\epsilon_2E_{1}$};
    \node (botf) at (2,-1) {$\epsilon_2C_{1}$};
    \draw (epse) -- (ee);
    \draw (epse-1) -- (ee-1);
    \draw (epse2) -- (e2);
    \draw (epse1) -- (e1);
    \draw (epsce) -- (fe);
    \draw (epsce-1) -- (fe-1);
    \draw (epsc2) -- (f2);
    \draw (epsc1) -- (f1);
    \draw (f1) -- (botf);
    \draw (e1) -- (bote);
  \end{tikzpicture} \end{array}
\end{equation}
(One just needs to check that all edges represent genuine incidences
between facets. In other words, that for every pair of facets sharing
an edge contain 3 common vertices.)

This covers all nodes in \eqref{eq:incidence} except $B$. Here are the
4 incidences around $B$:
\[
\begin{tikzpicture}[xscale=1.5,yscale=.9]
  \node (b) at (0,0) {$\begin{array}{c} B\\ {\tiny \verteq} \\ {\tiny \{ m_{++}, m_{+-}, p_+, a_{k+1} \}}\end{array}$};
  \node (se) at (0,3) {$\begin{array}{c} \sigma E_k \\ {\tiny \verteq}
                          \\ {\tiny \{ m_{++}, m_{+-}, a_{k+1},
                          \sigma(a_k) \} } \end{array}$};
  \node (fe) at (-3,0) {$\begin{array}{c} C_k\\ {\tiny \verteq} \\
                           {\tiny \{ m_{++},p_+, a_k, a_{k+1} \} } \end{array} $};
  \node (e4b) at (3,0) {$\begin{array}{c} \e_2B \\ {\tiny \verteq} \\
                           {\tiny \{m_{++}, m_{+-}, p_+,
                           \e_2a_{k+1}\}} \end{array}$};
  \node (e2f) at (0,-3) {$\begin{array}{c} \e_4C_k  \\ {\tiny \verteq} \\
                           {\tiny \{m_{+-}, p_+, a_{k+1}, a_k \} } \end{array}$};
  \draw (se) -- (b) -- (e2f);
  \draw (e4b) -- (b) -- (fe);
\end{tikzpicture}
\]

Hence we have checked that every facet appearing in Proposition
\ref{prop:top_skin} is incident to four facets, which are also
$\Gamma^+$-conjugates of facets on our list. Hence our list of facets
is complete.
\end{proof}

Below it will be useful to know the full incidence graphs of facets in
a fundamental domain, as well as their incidences to facets which
share a ridge with a facet in the fundamental domain. That is, are
``just outside'' the fundamental domain. Combining all
incidences discovered in the previous proof, we obtain the following
graph of incidences:
\begin{equation} \label{eq:full_incidence}
  \begin{array}{c}
\begin{tikzpicture}
  \node (e2b) at (4,5) {$\e_2B$};
  \node (ee) at (0,4) {$E_k$};
  \node (ee-1) at (0,3) {$E_{k-1}$};
  \node (edots) at (0,2) {$\vdots$};
  \node (e2) at (0,1) {$E_{2}$};
  \node (e1) at (0,0) {$E_{1}$};
    \node (fe) at (2,4) {$C_k$};
  \node (fe-1) at (2,3) {$C_{k-1}$};
  \node (fdots) at (2,2) {$\vdots$};
  \node (f2) at (2,1) {$C_{2}$};
  \node (f1) at (2,0) {$C_{1}$};
  \node (b) at (2,5) {$B$};
  \node (see) at (2,6) {$\sigma E_k$};
  \node (sb) at (0,5) {$\sigma B$};
  \draw (ee) -- (fe);
  \draw (ee-1) -- (fe-1);
  \draw (e2) -- (f2);
  \draw (e1) -- (f1);
  \draw (b) -- (fe);
  \draw (sb) -- (ee);
  \draw(ee) -- (ee-1);
  \draw(ee-1) -- (edots);
  \draw(edots) -- (e2);
  \draw(e2) -- (e1);
    \draw(fe) -- (fe-1);
  \draw(fe-1) -- (fdots);
  \draw(fdots) -- (f2);
  \draw(f2) -- (f1);
  \node (epse) at (-2,4) {$\epsilon_3C_k$};
  \node (epse-1) at (-2,3) {$\epsilon_3C_{k-1}$};
  \node (epse2) at (-2,1) {$\epsilon_3C_2$}; 
  \node (epse1) at (-2,0) {$\epsilon_3C_1$};
    \node (epsce) at (4,4) {$\epsilon_4C_k$};
  \node (epsce-1) at (4,3) {$\epsilon_4C_{k-1}$};
  \node (epsc2) at (4,1) {$\epsilon_4C_2$}; 
  \node (epsc1) at (4,0) {$\epsilon_4C_1$};
    \node (bote) at (0,-1) {$\epsilon_2E_{1}$};
    \node (botf) at (2,-1) {$\epsilon_2C_{1}$};
    \draw (epse) -- (ee);
    \draw (epse-1) -- (ee-1);
    \draw (epse2) -- (e2);
    \draw (epse1) -- (e1);
    \draw (epsce) -- (fe);
    \draw (epsce-1) -- (fe-1);
    \draw (epsc2) -- (f2);
    \draw (epsc1) -- (f1);
    \draw (f1) -- (botf);
    \draw (e1) -- (bote);
    \draw (b) -- (see);
    \draw (b) -- (e2b);
    \draw (b) -- (epsce);
  \end{tikzpicture} \end{array}
\end{equation}
Note that every facet in our fundamental domain (i.e. $B$, $C_i$ and
$E_i$) is incident to 4 other facets, and so we know this graph
describes all incidences amongst facets, one of which is in the
fundamental domain. (We do not need incidences between facets not in
the fundamental domain and have not included them above.)

\subsection{Quotients by sign flips} \label{sec:signflips}

Consider the facet-ridge graph of $D_k^+$. This is big and
complicated. However, if we quotient by $\Gamma^+$ we get something
rather simple, as can easily be seen using the computations of the
previous section. Sadly,
this quotient is too coarse, and we lose too much
information. However, by quotienting by the subgroup of $\Gamma^+$
generated by sign flips we 
get a graph which is both easily visualized, and powerful enough to
get bounds.

\begin{prop} \label{prop:fr}
  The quotient of the facet-ridge graph of $D_k^+$ by the action induced by sign
  flips yields the following graph:
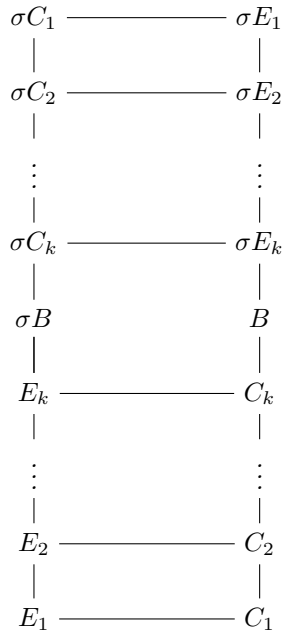
\[  \begin{array}{c}
\begin{tikzpicture}[xscale=1.5]
  \node (sc1) at (0,8) {$\sigma C_1$};
  \node (sc2) at (0,7) {$\sigma C_2$};
  \node (scdots) at (0,6) {$\vdots$};
  \node (sce) at (0,5) {$\sigma C_k$};
  \node (sB) at (0,4) {$\sigma B$};
  \node (ee) at (0,3) {$E_k$};
  \node (edots) at (0,2) {$\vdots$};
  \node (e2) at (0,1) {$E_{2}$};
  \node (e1) at (0,0) {$E_{1}$};
  \node (see) at (2,5) {$\sigma E_k$};
    \node (se1) at (2,8) {$\sigma E_1$};
  \node (se2) at (2,7) {$\sigma E_2$};
  \node (sedots) at (2,6) {$\vdots$};
  \node (B) at (2,4) {$B$};
  \node (fe) at (2,3) {$C_k$};
  \node (fdots) at (2,2) {$\vdots$};
  \node (f2) at (2,1) {$C_{2}$};
  \node (f1) at (2,0) {$C_{1}$};
    \draw (see) -- (sce);
    \draw (se2) -- (sc2);
    \draw (se1) -- (sc1);
    \draw (ee) -- (fe);
  \draw (e2) -- (f2);
  \draw (e1) -- (f1);
    \draw(sc1) -- (sc2);
    \draw(sc2) -- (scdots);
    \draw(scdots) -- (sce) -- (sB);
    \draw(sB) -- (ee);
  \draw(sB) -- (ee);
  \draw(ee) -- (edots);
  \draw(edots) -- (e2);
  \draw(e2) -- (e1);
      \draw(se1) -- (se2) -- (sedots) -- (see) -- (B) -- (fe);
    \draw(fdots) -- (f2);
    \draw(fdots) -- (fe);
    \draw(f2) -- (f1);
  \end{tikzpicture}
    \end{array} \]
(We ignore multiple edges and loops.)
\end{prop}

\begin{proof}
  Note that $\Gamma^+$ is isomorphic to a the semi-direct product of
  the subgroup of sign flips and $\langle \sigma \rangle$. In particular, by Proposition
  \ref{prop:top_skin} any facet is of the form $\e \sigma^j Y$ where
  $Y$ is one of the facets $B$, $C_i$ or $F_i$, $j \in \{ 0, 1\}$ and
  $\e$ is a product of sign flips. In particular, our graph contains
  representatives of all orbits under sign flips. It is now a simple
  matter to unpack the incidences listed in the proof of Proposition
  \ref{prop:top_skin} (or check the pictures on the previous page) to obtain the above graph.
\end{proof}

Actually, we need to consider a finer graph, which will almost
certainly appear unmotivated at present. We ask for a few pages' worth
of patience from the reader, where the relevance of this graph will
become apparent.

Consider the face graph $\Delta(D_k^+)$. We define $\widetilde{\GC}$
to be the full subgraph of $\Delta(D_k^+)$ with nodes corresponding to
the following faces of $\Delta(D_k^+)$:
\begin{enumerate}
\item All faces of dimension $\ge 2$;
\item All edges of the form $\{ a_i, a_{i+1}\}$ and their
  $\Gamma^+$-orbits.
\item Four vertices given by the $\Gamma^+$-orbit of $a_1$.
\end{enumerate}

We define $\GC$ to be quotient of $\widetilde{\GC}$ by the action of
sign flips.

\begin{prop}
  The graph $\GC$ is isomorphic to the graph displayed in Figure
  \ref{fig:G}.
\end{prop}

\begin{remark}
  In Figure \ref{fig:G} we indicate the dimension of the face by
  colouring the node: black
  nodes are of dimension $3$, blue nodes are of dimension 2, green
  nodes are of dimension 1 and the red nodes are of dimension 0. 
\end{remark}

\begin{proof}
  We obtain the full subgraph of $\GC$ corresponding to faces of
  dimension $\ge 2$ by taking the barycentric subdivision of the graph
  in Proposition \ref{prop:fr}. It remains to fill in the faces of
  dimensions $0$ and $1$ (the red and green nodes in Figure
  \ref{fig:G}). However, this is easy: $\{ a_i, a_{i+1}\}$
  (resp. $\sigma \{ a_i, a_{i+1}\}$) is incident only to $C_i$ and
  $E_i$ (resp. $\sigma(C_i)$ and $\sigma(E_i)$) by
  inspection. Similarly,  $a_1$ is only incident to $\{ a_1, a_2 \}$
  and $\sigma a_1$ is only incident to $\sigma \{ a_1, a_2\}$ and
  we're done.
\end{proof}

The following bound is critical to our main theorem. Its proof is by
inspection of Figure~\ref{fig:G}:

\begin{cor} \label{cor:bound}
  For any $F \in \{ B, \sigma B, E_i, C_i \}$ we have
  \[
\dist_{\GC}(F, \sigma a_1) \ge 2k + 3.
\]
Similarly, for any $F \in \{ B, \sigma B, \sigma E_i, \sigma C_i \}$
we have
  \[
\dist_{\GC}(F, a_1) \ge 2k + 3.
\]
\end{cor}

\begin{figure}
  \caption{The subgraph $\GC$.}
  \label{fig:G}
  \begin{tikzpicture}[xscale=3]
    \foreach \y in {1,3,...,17} {
         \node[blue]  (l\y) at (-1,\y) {$\bullet$}; 
         \node[blue]  (r\y) at (1,\y) {$\bullet$};
         \node[green]  (m\y) at (0,\y) {$\bullet$};
       }
        \foreach \y in {2,4,...,16} {
         \node[blue]  (m\y) at (0,\y) {$\bullet$}; 
       }
       \node (l2) at (-1,2) {$E_1$};
       \node (l4) at (-1,4) {$E_2$};
       \node (l6) at (-1,6) {$\vdots$};
       \node (l8) at (-1,8) {$E_k$};
       \node (l10) at (-1,10) {$\sigma B$};
       \node (l12) at (-1,12) {$\sigma C_k$};       
       \node (l14) at (-1,14) {$\vdots$};       
       \node (l16) at (-1,16) {$\sigma C_1$};
              \node (r2) at (1,2) {$C_1$};
       \node (r4) at (1,4) {$C_2$};
       \node (r6) at (1,6) {$\vdots$};
       \node (r8) at (1,8) {$C_k$};
       \node (r10) at (1,10) {$ B$};
       \node (r12) at (1,12) {$\sigma E_k$};       
       \node (r14) at (1,14) {$\vdots$};       
       \node (r16) at (1,16) {$\sigma E_1$};
       \node[red] (m0) at (0,0) {$a_1$};
       \node[red]  (m18) at (0,18) {$\sigma a_1$};
       \foreach \y in {1,2,...,17} {
         \draw (l\y) -- (m\y);
         \draw (m\y) -- (r\y); }
       \foreach \y in {1,2,...,16} {
         \pgfmathtruncatemacro{\lp}{\y+1};
         \draw (l\y) -- (l\lp); \draw (m\y) -- (m\lp); \draw (r\y) --
         (r\lp); }
       \draw (m0) -- (m1);
       \draw (m17) -- (m18);
       \node (blah) at (0,19) {{\color{white} $aaa$}};
       \draw[draw=white, fill=white] (-.9,9.2) rectangle (.9,10.8);

     \end{tikzpicture}
   \end{figure}

\subsection{The pair embedding}

Now we focus on $D = D_k$ with $k \ge 1$ as above.

Recall the full subgraphs $\Delta_D(D^-) \subset \Delta(D^-)$ and
$\Delta_{D}(D^+) \subset \Delta(D^+)$ defined in \S \ref{sec:pair}:
they consists of facets which occur as the intersection of any face of
$F$ with the bottom (resp. top) drum skin.
We would like to describe $\Delta_{D}(D^-)$ and
$\Delta_{D}(D^+)$. This is
tricky, but it turns out that we can describe a graph which is (potentially)
slightly bigger than $\Delta_{D}(D^-)$, but is small enough to
estimate distances.

The bottom and top drum skins are isomorphic (via $\tau$). However,
there are many possible isomorphisms. We make the fixed and arbitrary
choice to use $\tau$ to identify the top and bottom drum skins. We
will use superscript to indicate application of $\tau$. This applies
to any face of $D_k^+$. So for example
\[
a_1^- = \tau(a_1) \quad \text{and} \quad C_i^- = \tau(C_i).
  \]

  \begin{remark}
    This identification of vertices in the top and bottom drum skin is
    achieved in the colab \cite{colab} \verb|permute_axes()|. Be
    careful that one should only use this function to identify
    vertices in the top of drum skin with those in the bottom, and not
    vice versa! The
    identification of facets in the top and bottom drum skin is
    achieved using the dictionary  \verb|top_to_bottom_dict|. \end{remark}

  Below a crucial role will be played by $\sigma(a_1)$ (the
  ``north pole'' in Figure~\ref{fig:G}). To simplify notation we
  denote this point by $n$:
  \[
n = \sigma(a_1) = (0,100,0,0,1).
\]
Similarly
\[
  n^- = \tau(n) = (0,0,0,100,-1)
\]
  is the corresponding point in the bottom
drum skin.  We have:

  \begin{thm} \label{thm:pair}
    The facet-vertex map is given on the top drum skin as follows:
    \begin{align*}
      \phi^+(B) = a_1^-, \quad
      \phi^+(C_i) = n^- \quad \text{and} \quad
      \phi^+(E_i) =  n^-.
    \end{align*}
    On the bottom drum skin it is given by
    \begin{align*}
      \phi^-(B^-) =  n, \quad 
      \phi^-(C_i^-) = a_1 \quad \text{and} \quad
      \phi^-(E^-_i) =  a_1 
    \end{align*}
  \end{thm}

  \begin{remark}
    If we had to try to isolate one key fact that makes our examples
    work, it would be Theorem \ref{thm:pair}.
  \end{remark}

\begin{remark} \change{ \label{rem:doublethm}
  As pointed out by a referee, the ``double pancake'' structure remarked
upon in Remark \ref{rem:double pancake} goes some way to explaining the simple form of
the facet vertex map in Theorem \ref{thm:pair}. Indeed, consider a polytope
$_\e D_k^+$ obtained by multiplying the first two coordinates by
$\e$, where $\e$ is roughly $1/50$. Then $_\e D_k^+$ is
combinatorially equivalent to $D_k^+$, and the coordinates of its
vertices are all of a similar order of magnitude. Thus, we expect the
directions of the normals to facets to be reasonably well-distributed
on the $3$-sphere. However, as $\e \to 1$ the first two coordinates
are scaled rather drastically, and these normal
vectors approach the unit circle in the $3$-sphere spanned by the
last two coordinates. This explains why only the ``extremal'' vertices
occur in the image of the facet vertex map in Theorem \ref{thm:pair}. This
scaling technique was observed (and used for a similar purpose) in
\cite[Lemma 2.13]{MSW}.}
\end{remark}

  \begin{remark}
    The reader wishing to witness this little miracle in the colab \cite{colab}
    should see the cells labelled \verb|Facet-vertex map for top facets|
    for $\phi^+(B),  \phi^+(C_i)$ and $\phi^+(E_i)$ and
    \verb|Facet-vertex map for bottom facets| for $\phi^-(B^-),
    \phi^-(C_i^-)$ and $\phi^-(E_i^-)$. 
  \end{remark}

  \begin{proof} We will use knowledge obtained in the course of the
    proof of Proposition \ref{prop:top_skin}. As in that proof, we
    proceed facet by facet.

   We will make essential use of the method outlined in \S\ref{sec:fv} to
   determine the facet-vertex map.

   \emph{Facet $B$:} We have in the proof of Proposition
   \ref{prop:top_skin} that
  \[f_B =  (-1, -\frac{24}{25}, -49, 0, 3 \cdot 49) \]
yields a functional which cuts out $B$ from our top drum skin. By
Lemma~\ref{lem:minimum}, in order to determine $\phi^+(B)$, we only need to determine the vertex $v$ in
the bottom drum skin on which $f_B$ takes its minimum. By inspection,
this occurs at
\[
v = (0, 0, 100, 0).
\]
That is, $\phi^+(B) = a_1^-$ as claimed.

   \emph{Facet $E_i$:} As in the proof of Proposition
   \ref{prop:top_skin} for $E_i$, let us split $\RM^4 = \RM^2 \oplus
   \RM^2$ and consider the same functional $f = f_1 + f_2 + 1$ as
   there. Note that
\[
f_1((75,75)) \ge -1
\]
and hence (using that $f_1$ is $-1$ on $a_i$ and $a_{i+1}$) that
\begin{align} \label{eq:25}
f_1(v) &\ge -1/25 \quad \text{for $v = (\pm 3, \pm 3)$, $(\pm 1, 0)$ or $(0, \pm
    1)$} \\
  f_1((1,0))   &\ge -1/75. \label{eq:75}
\end{align}
Also, we have $f_2((c,d)) = -\frac{1}{3} d$. Using this explicit form
of $f_2$ and \eqref{eq:25} it is clear that $f$ takes its minimum
value at either
\[
(0,0,0,100) \quad \text{or} \quad (1,0,0,98).
\]
Now $-100/3 < -98/3 - 1/75$ and hence $f$ takes its minimum of the
first point, i.e. $n^-$. Thus $\phi^+(E_i) = n^-$ as claimed.

   \emph{Facet $C_i$:} Again, we make crucial use of knowledge gleaned in the course of the
   proof of Proposition \ref{prop:top_skin}. We decompose $f = f_1 +
   f_2 + 1$ as in Proposition \ref{prop:top_skin}. We saw there that
   \begin{equation}
  \label{eq:f2b}
  f_2 = (\delta_2, -\frac{1}{3}- \delta_2).
\end{equation}
where $\delta_2$ satisfies
\[
-0.04 < \delta_2 \le -0.02 
\]
From the form of $f_2$ it follows that
amongst the $\Gamma^+$ orbits of $\{ a_i^- \}$, $f$ takes its minimum
value at $n^-$. It is also clear that amongst the $\Gamma^+$-orbits of
$p_+^-$, $f$ takes its minimum value at $(1,0,0,98)$. Finally, we
claim that
\[
f((0,0,0,100)) < f((1,0,0,98)).
\]
Expanding and rearranging, we see that this is true if and only if
we have the inequality:
\[
  2\cdot f_2((0,1)) < f_1((1,0)).
\]
By \eqref{eq:ASSUMP} we have $f_1((1,0)) > -\frac{1}{100}$, and by
\eqref{eq:f2b} we have $f_2((0,1))  = -\frac{1}{3}- \delta_2$. Thus
the inequality holds, and we have $\phi^+(C_i) = n^-$ as claimed.

This completes the proof of Theorem~\ref{thm:pair} for facets $B$, $C_i$ and $E_i$ in the
top drum skin.

For those in the bottom drum skin we employ a symmetry
argument. Because $\phi^+(B) = a_1^-$ there exists a facet $F$ of the
drum $D_k$ incident to both $B$ and $a_1^-$. Acting by $\tau$ we see
that there is a facet $\tau(F)$, incident to both $B^-$ and
$\tau(a_1^-) = \tau^2(a_1) = \sigma(a_1) = n$. This proves that
$\phi^-(B^-) = n$ as claimed. The statements $\phi^-(C^-_i) =
\phi^-(E^-_i) =\tau(n^-) = \tau^2(n) = \sigma(n) = a_1$ follows by the same
argument. The theorem is proved. 
\end{proof}

In order to establish our main theorem, we need to know a little more
about the pair embedding:

  \begin{thm} \label{thm:pair2} Only boundary edges occur in the image
    of the pair embedding. More precisely, suppose that $E$ is an edge of $D_k^-$ and that $\rho(F) = (X,E)$ for
some facet $F$ of $D_k$. Then $E$ belongs to the $\Gamma^+$-orbit of
$\{ a_i^-, a_{i+1}^- \}$ for some $i$.
\end{thm}

The next section is devoted to the proof.

\subsection{The pair embedding II}
Consider the incidence graph of facets incident to the fundamental
domain in \eqref{eq:full_incidence}. Using Theorem \ref{thm:pair} and
the fact that the facet-vertex map is equivariant one computes easily
that the images of the facets displayed in \eqref{eq:full_incidence}
are given as follows:
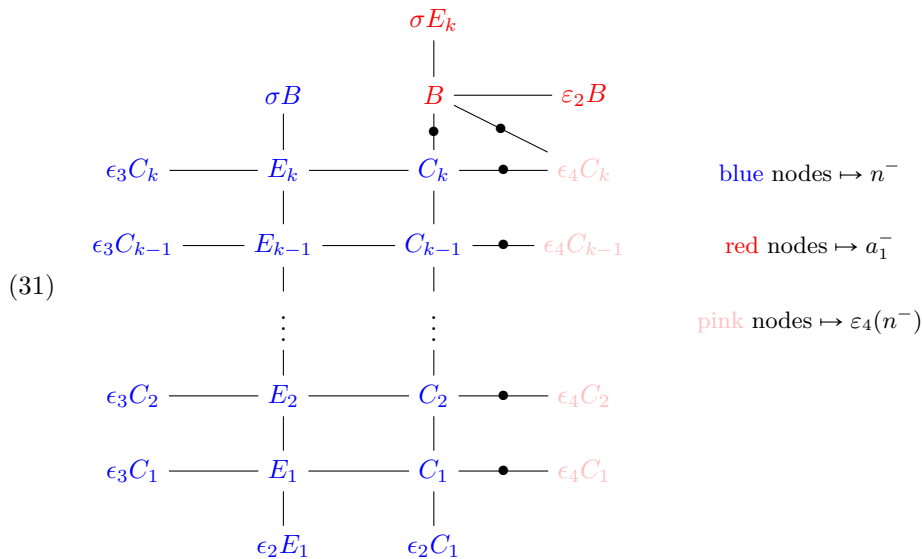
\begin{equation} \label{eq:full_incidence_fvmap}
  \begin{array}{c}
\begin{tikzpicture}[scale=1]
  \node[red] (e2b) at (4,5) {$\e_2B$};
  \node[blue] (ee) at (0,4) {$E_k$};
  \node[blue] (ek-1) at (0,3) {$E_{k-1}$};
  \node (edots) at (0,2) {$\vdots$};
  \node[blue] (e2) at (0,1) {$E_{2}$};
  \node[blue] (e1) at (0,0) {$E_{1}$};
    \node[blue] (fe) at (2,4) {$C_k$};
  \node[blue] (fk-1) at (2,3) {$C_{k-1}$};
  \node (fdots) at (2,2) {$\vdots$};
  \node[blue] (f2) at (2,1) {$C_{2}$};
  \node[blue] (f1) at (2,0) {$C_{1}$};
  \node[red] (b) at (2,5) {$B$};
  \node[red] (see) at (2,6) {$\sigma E_k$};
  \node[blue] (sb) at (0,5) {$\sigma B$};
  \draw (ee) -- (fe);
  \draw (ek-1) -- (fk-1);
  \draw (e2) -- (f2);
  \draw (e1) -- (f1);
  \draw (b) to node {$\bullet$} (fe);
  \draw (sb) -- (ee);
  \draw(ee) -- (ek-1);
  \draw(ek-1) -- (edots);
  \draw(edots) -- (e2);
  \draw(e2) -- (e1);
    \draw(fe) -- (fk-1);
  \draw(fk-1) -- (fdots);
  \draw(fdots) -- (f2);
  \draw(f2) -- (f1);
  \node[blue] (epse) at (-2,4) {$\epsilon_3C_k$};
  \node[blue] (epsk-1) at (-2,3) {$\epsilon_3C_{k-1}$};
  \node[blue] (epse2) at (-2,1) {$\epsilon_3C_2$}; 
  \node[blue] (epse1) at (-2,0) {$\epsilon_3C_1$};
    \node[pink] (epsce) at (4,4) {$\epsilon_4C_k$};
  \node[pink] (epsck-1) at (4,3) {$\epsilon_4C_{k-1}$};
  \node[pink] (epsc2) at (4,1) {$\epsilon_4C_2$}; 
  \node[pink] (epsc1) at (4,0) {$\epsilon_4C_1$};
    \node[blue] (bote) at (0,-1) {$\epsilon_2E_{1}$};
    \node[blue] (botf) at (2,-1) {$\epsilon_2C_{1}$};
    \draw (epse) -- (ee);
    \draw (epsk-1) -- (ek-1);
    \draw (epse2) -- (e2);
    \draw (epse1) -- (e1);
    \draw (epsce) to node {$\bullet$} (fe);
    \draw (epsck-1) to node {$\bullet$} (fk-1);
    \draw (epsc2) to node {$\bullet$} (f2);
    \draw (epsc1) to node {$\bullet$} (f1);
    \draw (f1) -- (botf);
    \draw (e1) -- (bote);
    \draw (b) to (see);
    \draw (b) to (e2b);
    \draw (b) to node {$\bullet$} (epsce);
    \node at (7,4) {\small $\text{{\color{blue}blue} nodes} \mapsto n^-$};
  \node at (7,3) {\small $\text{{\color{red}red} nodes}\mapsto a_1^-$};
    \node at (7,2) {\small $\text{{\color{pink}pink} nodes}\mapsto
      \e_4(n^-)$};
\end{tikzpicture} \end{array}
\end{equation}
More precisely, the images of $B, C_i$ and $E_i$ are read off
directly from Theorem \ref{thm:pair}. To compute the remaining images
one uses equivariance. For example, $C_k$ maps to $n^-$ by Theorem
\ref{thm:pair} and hence $\e_4C_4$ maps to $e_4(n^-)$.

Recall that edges in the above graph correspond to certain ridges in
$D_k^+$. We have indicated edges where the facet-vertex map takes
different values at the facets corresponding to the end points by adding a bullet to the edge. We
can then apply Corollary \ref{eq:ridges}  
to deduce that we only need to determine the facet-ridge map for
facets in the full drum which involve these ridges. 

We handle these ridges one at a time:

\emph{Ridge $R = B \cap C_k$:} 

We apply Lemma \ref{lem:pair prog} with a minor
variation. We determine all functionals which are zero on $R$ and
$< 0$ on $D_k^+ - R$ and we compute their maxima on $D_k^-$. (We have
switched $\ge 0 \leftrightarrow \; \le 0$ and maxima $\leftrightarrow$
minima with respect to Lemma \ref{lem:pair prog}. This simplifies some of the
constants below.)  It is enough to show that all maxima are in the
$\Gamma^+$-orbit of $a_i^-$. (\change{We use that the only edges
involving only the $a_i^-$  are $\Gamma^+$-orbits of edges of the form
$(a_i^-,a_{i+1}^-)$. This is an easy consequence of symmetry and
Proposition \ref{prop:top_skin}.})

Consider the functional given by
\[
f = (x_1, x_2, x_3, x_4, x_5) \quad \text{i.e. $f((v_1,\dots,v_4)) = \sum
  x_iv_i + v_5$.}
\]
We now determine the conditions for $f$ to vanish on our ridge $R$ and be
$\le 0$ on $D_k^+$. Certainly, $x_5 \ne 0$ as $D_k^+$ contains $0$ in
its interior. Hence we may assume that
\begin{equation}
  \label{eq:x5}
  x_5 = -150.
\end{equation}
Now $f$ vanishes at $a_{k+1}$ and is $\le 0$ on the $\Gamma^+$-orbit
of the $a_i$. Writing out what that means for $a_{k+1}$ and the
$\Gamma^+$ orbit of $a_1$ we deduce the inequalities
\begin{equation}
  \label{eq:x1x2}
  x_1 + x_2 = 2 \quad \text{and} \quad \frac{1}{2} \le x_i \le
  \frac{3}{2} \text{ for $i = 1,2$.}
\end{equation}
Because $f$ vanishes at $m_{++}$ and is $\le 0$ on $m_{\pm\pm}$ we
deduce that
\begin{equation}
  \label{eq:x3x4}
  x_3 + x_4 = 50, \quad x_3, x_4 \ge 0.
\end{equation}
Finally, because $f$ takes the value
$0$ at $p_+$ and is $< 0$ at $\sigma(p_+)$ we deduce that
\begin{equation}
  \label{eq:x1x3}
  x_3 = 150 - 98x_1 \quad \text{and} \quad
  x_4 < 150 - 98x_2.
\end{equation}
Combining \eqref{eq:x1x3}, \eqref{eq:x3x4} and \eqref{eq:x1x2} we deduce that
\begin{equation} \label{eq:x1x2x3x4}
  \frac{50}{49} \le x_1 \le \frac{3}{2},   \quad \frac{1}{2} \le x_2
  \le \frac{48}{49},  \quad \quad 3 \le x_3 \le 50.
\end{equation}

It remains to compute the maximum of $f$ on $D_k^-$.

We will now show that $f$ obtains its maximum on the $\Gamma^+$-orbit
of some $a^-_i$. In order to do this, it is enough to show that $f$
\change{obtains its maximum} neither on the $\Gamma^+$-orbit of
$m_{\pm\pm}^-$ nor the orbit of $p_{\pm }^-$. Firstly note that on
each orbit it takes its maximum value on vectors all of whose
coordinates are positive, i.e. $m_{++}^-$ or $p_+^-$ and
$\sigma(p_+^-)$ respectively.

Now, at $m_{++}^-$, $f$ takes the value $3(x_1 + x_2) = 6$, which is much less
than its value at $a_k = 50 \cdot 75$.

It is slightly more subtle to see the statement for the $p_+^-$ and
$\sigma(p_+^-)$. Firstly, suppose that $x_3 \ge x_4$. Then $x_3$ is at
least 25 by \eqref{eq:x3x4} and hence
\[
f(\sigma(p_+)^-) \le f(p_+^-) = x_2 + 98x_3 < 100x_3 = f(a_1^-).
\]
Hence in this case neither $f(p_+^-)$ nor $f(\sigma(p_+)^-)$ is maximal.
On the other hand, if $x_3 \le x_4$ then $x_4$ is at least 25 and
\[
f(p_+^-) \le f(\sigma(p_+)^-) = x_1 + 98x_4 < 100x_4 = f(n^-).
\]
Hence in this case, again neither $f(p_+^-)$ nor $f(\sigma(p_+)^-)$ is
maximal. This completes the proof.


\emph{Ridge $R = B \cap \e_4C_k$:} We can handle this by symmetry: $\e_4$
fixes $B$ and sends $C_k$ to $\e_4C_k$. Thus this case follows from
the previous one, as $\e_4$ permutes boundary edges.

\emph{Ridge $R = C_i \cap \e_4C_i$:} The ridge $R$ has vertices $a_i, a_{i+1}$ and
$p_+$. We use the same method as in the first case handled above:
first we write down conditions satisfied by all functionals which are
$=0$ on $R$ and $<0$ on $D^+_k - R$, and then show that the maxima on
$D_k^-$ is contained in a boundary vertex or edge.

As before consider the functional
\[
f = (x_1, x_2, x_3, x_4, x_5) \quad \text{i.e. $f((v_1,\dots,v_4)) = \sum
  x_iv_i + v_5$.}
\]
and assume that $f$ is $=0$ on $R$ and $<0$ on $D^+_k- R$. Certainly $x_5
\ne 0$ as $0$ is in the interior of $D_k^+$ and we can (and do) assume
that
\begin{equation}
  \label{eq:x5}
  x_5 = -102.
\end{equation}
The equation $x_1u + x_2v - 102$ (the restriction of our functional to
the span of the first two coordinates) determines a functional which
vanishes on the two points $a_i$ and $a_{i+1}$. We deduce that the
$u$- (resp. $v$-) intercept is given by
\[
u = 102/ x_1 \quad \text{resp. $v = 102/x_2$}.
\]
Our assumptions \eqref{eq:ASSUMP} and \eqref{eq:ASSUMPy} then imply
that
\begin{equation}
  \label{eq:x1}
  1 \le x_1 \le \frac{102}{100} = 1.02
\end{equation}
and
\begin{equation}
  \label{eq:x2}
  0 < x_2 < 102/283 \approx 0.36.
\end{equation}
Because our function vanishes on $p_+$ we have
\begin{equation}
  \label{eq:x3}
  x_3 = 102 - 98x_1 \quad \text{and thus} \quad 2.04 < x_3 \le 4.
\end{equation}

We already have enough to deduce what we need: \eqref{eq:x1},
\eqref{eq:x2} and \eqref{eq:x3} are enough to deduce that amongst the
points $m_{\pm,\pm}^-$, $p_{\pm}^-$ and $a_i^-$ our functional $f$
obtains its maximum at the orbit of one of the $a_i^-$, which is what
we were hoping to prove.

\begin{remark}
  With a little more care, one can establish that in the case of the
  ridge $B \cap C_k$, all points
\[
a^-_1, a^-_2,
\dots, a^-_{k-1}, a^-_k, \sigma(a^-_k), \dots, \sigma(a^-_2),
\sigma(a^-_1)
\]
can be obtained as maxima of a suitable $f$.

Similarly, one can also see that in case of the ridge $R = B \cap
\e_4C_i$ all points
\[
-\sigma(a^-_1), -\sigma(a^-_2), \dots, -a_2^-, a^-_1, a^-_2,
\dots, a^-_{k-1}, a^-_k, \sigma(a^-_k), \dots, \sigma(a^-_2),
\sigma(a^-_1)
\]
can be obtained as maxima of a suitable $f$.

One can see an illustration of this phenomenon in the colab \cite{colab} in the
cell: \\ \verb|Edges incident to a fixed ridge in the top drum skin.|
\end{remark}

\subsection{Proof of our main theorem} We now have all
the pieces in place to explain our proof of Theorem \ref{thm:main}.

As explained in \S\ref{sec:pair}, a priori we need to compute the
distance between the top and bottom drum skin in the facet ridge graph
$\FC(D_k)$. Instead, we can compute distances in the trimmed facet
ridge graph $\FC^t(D_k)$. In order to do this we can try to use the
pair embedding
\[
\FC^t(D_k) \into \Delta(D^-) \times \Delta(D^+).
\]
Unfortunately, the bounds we get in this way are not strong enough.

Recall the graphs $\widetilde{\GC}$ and $\GC$ introduced in
\S~\ref{sec:signflips}. We let  $\widetilde{\GC}^{-}$ and $\GC^{-}$
denote their obvious analogues for the bottom drum skins. By Theorems
\ref{thm:pair} and \ref{thm:pair2} the only $0$ and $1$-dimensional
facets in the image of the pair embedding are the $\Gamma$ orbits of
the point $a_1$ and the edges $\{ a_i, a_{i+1} \}$. Thus, by
definition of $\widetilde{\GC}$ and $\widetilde{\GC}^-$ we instead obtain
an embedding
\[
\FC^t(D_k) \into \widetilde{\GC}^{-} \times \widetilde{\GC}.
\]
Quotienting by sign flips we obtain an embedding
\begin{equation}\label{eq:Gsign}
\FC^t(D_k) / \text{sign flips} \into \GC^{-} \times \GC.
\end{equation}

Finally, note that by Corollary \ref{cor:bound} we have the following
bounds on distances in $\GC$:
\begin{align}
  \dist_{\GC} ( B, n) & \ge 2k + 3, \label{eq:b1} \\
  \dist_{\GC} ( B, a_1) & \ge 2k + 3, \label{eq:b2} \\
  \dist_{\GC} ( X, n) & \ge 2k + 3 \quad \text{for any $X \in \{ C_i,
                          E_i\}$}. \label{eq:b3} 
\end{align}
We also have analogues of the above for $\GC^-$, for example
$\dist_{\GC^-} ( B^-, n^-) \ge 2k + 3$. 

\begin{thm}
  $D_k$ has width $\ge 5 + k$.
\end{thm}

In the proof, the facets $B$ and $\{ C_i, E_i \}$ play very different
roles. In order to lighten notation, we use $X^-$ (resp. $X$) to denote
any element of the set $\{ C_i^-, E_i^- \}$ (resp. $\{ C_i, E_i\}$). 

\begin{proof} Our drum $D_k$ has width $\ge 5 + k$ if and only if any
  path $F_0, F_1, \dots, F_\ell$ of facets in the trimmed facet ridge
  graph of $D_k$ has length (see \eqref{eq:width2}):
  \begin{equation}
\ell \ge 3 + k\label{eq:target_bound}
\end{equation}
For the rest of the proof we will use the pair embedding to regard
each $F_i$ as a pair of faces in the top and bottom drum 
skin. We further pass to the quotient by sign flips and regard facets
as pairs of nodes in $\GC^- \times \GC$ via \eqref{eq:Gsign}. Using
the symmetry $\sigma$ of $\GC^{-1}$ we can additionally
assume that that $(F_0)_\bot$ is either $B^-$ or $X^-$ (see Proposition
\ref{prop:top_skin}). We analyse these cases separately:

\emph{Case 1: $(F_0)_\bot = B^-$.} By Theorem \ref{thm:pair} we must have
\[
F_0 = (B^-, n).
\]
Applying Theorem \ref{thm:pair} again, we conclude that either
\[
F_\ell = (n^-, X) \quad \text{or} \quad F_\ell = ( a_1^-, \sigma(B)).
\]

If $F_\ell = (n^-, X)$ then we have
\[
  \dist(F_0, F_\ell) \ge \frac{1}{2} ( \dist_{\GC}(B^-, n^-) +
\dist_{\GC}(n, X)) \stackrel{\eqref{eq:b1}}{\ge} \frac{1}{2} ( (2k +
3) +
3) = k + 3
\]
where the first inequality is a consequence of Proposition \ref{prop:pair_bound}.

If $F_\ell = ( a_1^-, \sigma(B))$ then we have
\[
  \dist(F_0, F_\ell) \ge \frac{1}{2} ( \dist_{\GC}(B^-, a_1^-) +
\dist_{\GC}(n, \sigma(B))) \stackrel{\eqref{eq:b2}}{\ge} \frac{1}{2} ( (2k +
3) +
3) = k + 3
\]
where again we use Proposition \ref{prop:pair_bound} for the first
inequality.

\emph{Case 2: $(F_0)_\bot = X^-$.} By Theorem \ref{thm:pair}, we must
have
\[
F_0 = (X^-,a_1).
\]
Applying Theorem \ref{thm:pair} again, we also conclude that either
\[
F_\ell = (n^-, X) \quad \text{or} \quad F_\ell = ( a_1^-, \sigma(B)).
\]
If $F_\ell = (n^-, X)$ then we have
\[
  \dist(F_0, F_\ell) \ge \frac{1}{2} ( \dist_{\GC}(X^-, n^-) +
\dist_{\GC}(a_1, X)) \stackrel{\eqref{eq:b3}}{\ge} \frac{1}{2} ( (2k +
3) +
3) = k + 3.
\]
If $F_\ell = ( a_1^-, \sigma(B))$ then we have
\[
  \dist(F_0, F_\ell) \ge \frac{1}{2} ( \dist_{\GC}(X^-, a_1^-) +
\dist_{\GC}(a_1, \sigma(B))) \stackrel{}{\ge}
\frac{1}{2} ( 3 + (2k +
3)) = k + 3
\]
where, for the second inequality, we have used that
\[ \dist_{\GC}(a_1, \sigma(B))) = \dist_{\GC}(n, B)) \ge 2k + 3 \]
by symmetry and \eqref{eq:b1}.

Thus we have established \eqref{eq:target_bound} in all cases, and the theorem is proved.
\end{proof}


\begin{thebibliography}{MSW15}

\bibitem[Co]{colab}
Colab for ``Drums of High Width''.\\
\newblock
  tiny.cc/1ioivz
\newblock Accessed: 2023-12-15.


\bibitem[HW10]{sym}
M.~Heule and T.~Walsh.
\newblock Symmetry within solutions.
\newblock In {\em Proceedings of AAAI}, volume~10, pages 77--82, 2010.

\bibitem[KW67]{KleeWalkup}
V.~Klee and D.~W. Walkup.
\newblock The {$d$}-step conjecture for polyhedra of dimension {$d<6$}.
\newblock {\em Acta Math.}, 117:53--78, 1967.

\bibitem[MSW15]{MSW}
B.~Matschke, F.~Santos, and C.~Weibel.
\newblock The width of five-dimensional prismatoids.
\newblock {\em Proc. Lond. Math. Soc. (3)}, 110(3):647--672, 2015.

\bibitem[San12]{Santos}
F.~Santos.
\newblock A counterexample to the {H}irsch conjecture.
\newblock {\em Ann. of Math. (2)}, 176(1):383--412, 2012.

\end{thebibliography}

  \end{document}